\documentclass[12pt]{amsart}
\usepackage{amsmath,amssymb,latexsym}
\usepackage{fullpage}
\usepackage{amscd}
\theoremstyle{plain}
\newtheorem{theorem}{Theorem}

\newtheorem{lemma}{Lemma}

\newtheorem{conjecture}{Conjecture}

\numberwithin{equation}{section}

\begin{document}

\title{Integrals derived from the doubling method}
\author{David Ginzburg}

\address{School of Mathematical Sciences, Sackler Faculty of Exact Sciences, Tel-Aviv University, Israel
	69978} \email{ginzburg@post.tau.ac.il}

\thanks{This research was supported by the ISRAEL SCIENCE FOUNDATION
	(grant No. 461/18).}

\author{David Soudry}
\address{School of Mathematical Sciences, Sackler Faculty of Exact Sciences, Tel-Aviv University, Israel
	69978} \email{soudry@post.tau.ac.il}





\begin{abstract}
In this note, we use a basic identity, derived from the generalized doubling integrals of \cite{C-F-G-K1}, in order to explain the existence of various global Rankin-Selberg integrals for certain $L$-functions. To derive these global integrals,  we use the identities relating Eisenstein series in \cite{G-S}, together with the process of exchanging roots. We concentrate on several well known examples, and explain how to obtain them from the basic identity. Using these ideas, we also show how to derive a new global integral.

\end{abstract}

\maketitle

\section{Introduction}
The first construction of global Rankin-Selbrg integrals using the doubling method is due to Piatetski-Shapiro and Rallis. See \cite{PS-R1} where the authors introduced a global integral which represents the standard $L$-function attached to any irreducible, automorphic,  cuspidal representation $\pi$ of $G(\bf A)$, where $G$ is a classical group defined over a number field $F$, and $\bf A$ is its ring of Adeles. We recall this briefly for the symplectic group $G=Sp_{2k}$. Let $\pi$ denote an irreducible, automorphic, cuspidal representation of $Sp_{2k}({\bf A})$. Form the Eisenstein series $E(f_s)$ on $Sp_{4k}({\bf A})$, associated to a smooth, holomorphic section $f_s$ of the parabolic induction $Ind_{Q_{2k}({\bf A})}^{Sp_{4k}({\bf A})}|\det\cdot|^s$, where $Q_{2k}\subset Sp_{4k}$ is the Siegel parabolic subgroup. The doubling integral introduced in \cite{PS-R1} is given by
\begin{equation}\label{dou1}
\int\limits_{Sp_{2k}(F)\times Sp_{2k}(F)\backslash Sp_{2k}({\bf A})\times Sp_{2k}({\bf A})}\varphi_1(g)\overline{\varphi_2(h)}E(f_s)(t(g,h))dgdh,
\end{equation}
where, $\varphi_i$, $i=1,2$, are cusp forms in the space of $\pi$, and $t(g,h)$ denotes the direct sum embedding of $Sp_{2k}({\bf A})\times Sp_{2k}({\bf A})$ inside $Sp_{4k}({\bf A})$. This integral represents $L(\pi,s+\frac{1}{2})$, after normalizing the Eisenstein series. A similar construction exists for general linear groups and for orthogonal groups. 

In recent years new constructions of global integrals which use the doubling method were discovered. First, in \cite{G-H}, the authors constructed a doubling integral for the exceptional group $G_2$. This integral, which involves an Eisenstein series on the exceptional group $E_8$, represents the seven degree $L$-function. Recently, in \cite{C-F-G-K1}, the method of \cite{PS-R1} was extended to a more general doubling integral, which represents the standard $L$-function $L(\pi\times\tau,s+\frac{1}{2})$, where $\pi$, $\tau$ are irreducible, automorphic, cuspidal representations of $G(\bf A)$, $GL_n({\bf A})$ respectively, and $G$ is a split classical group. As explained in \cite{C-F-G-K2}, this construction can be extended to represent $L$-functions $L(\tilde{\pi}\times\tilde{\tau},s+\frac{1}{2})$ for pairs of irreducible, automorphic, cuspidal representations of metaplectic covers of $G(\bf A)$ and $GL_n({\bf A})$. All the above integrals have the same structure in the sense that they have the form
\begin{equation}\label{dou2}
\int\limits_{G(F)\times G(F)\backslash G({\bf A})\times G({\bf A})}\ 
\int\limits_{U(F)\backslash U({\bf A})}
\varphi_1(g)\overline{\varphi_2(h)}E(f_{\mathcal{E}_\tau,s})(ut(g,h))\psi^{-1}_U(u)dudgdh.
\end{equation}
Here, $\varphi_i$, $i=1,2$, are cusp forms in the space of $\pi$; $E(f_{\mathcal{E}_\tau,s})$ is a certain Eisenstein series on a certain group $H$, attached to a smooth, holomorphic section $f_{\mathcal{E}_\tau,s}$ of a parabolic induction on $H(\bf A)$, with parabolic data $(P,\mathcal{E}_\tau)$, where $\mathcal{E}_\tau$ is a certain automorphic representation of the Levi part of $P(\bf A)$, associated to $\tau$; $U$ is a certain unipotent subgroup of $H$ and $\psi_U$ is a character of $U(\bf A)$, trivial on $U(F)$. Finally, $t$ denotes an embedding of $G({\bf A})\times G({\bf A})$ inside $H(\bf A)$, so that conjugation by $t(g,h)$ preserves $U(\bf A)$ and $\psi_U$.  For $G=Sp_{2k}$, integral \eqref{dou1} is obtained from integral \eqref{dou2} by taking $H=Sp_{4k},\ P=Q_{2k},\ U$- the trivial group, $\tau$- the trivial representation of $GL_1({\bf A})$, and $\mathcal{E}_\tau$- the trivial representation of $GL_{2k}({\bf A})$. 

In all known cases of the doubling method \cite{PS-R1}, \cite{G-H}, \cite{C-F-G-K1}, starting with integral \eqref{dou2}, for $\text{Re}(s)$ large, after an unfolding process it is equal to
\begin{equation}\label{dou4}
\int\limits_{ G({\bf A})}
\int\limits_{ U_0({\bf A})}
<\varphi_1,\pi(h)\varphi_2>f_{W({\mathcal E}_\tau),s}(\delta u_0t(1,h))\psi^{-1}_U(u_0)du_0dg.
\end{equation}
Here, $<\varphi_1,\pi(h)\varphi_2>$ denotes the standard $L^2$- inner product of the two cusp forms $\varphi_1,\pi(h)\varphi_2$; $U_0$ is a certain unipotent subgroup of $U$; $\delta$ is a certain element in $H(F)$; $f_{W({\mathcal E}_\tau),s}$ is the composition of the section with a functional $W({\mathcal E}_\tau)$ on $\mathcal{E}_\tau$ given by a certain Fourier coefficient. It defines a unique (up to scalars) functional, satisfying appropriate equivariance properties, on the space of ${\mathcal E}_\tau$. For example, in the cases considered in \cite{PS-R1} and \cite{G-H}, ${\mathcal E}_\tau$ is the trivial representation. In the case of \cite{C-F-G-K1}, ${\mathcal E}_\tau$ is a Speh representation $\Delta(\tau,m)$ (of appropriate length $m$), associated with $\tau$. For example, for $G=Sp_{2k}$, $m=2k$, and $P=Q_{2nk}$ is the Siegel parabolic subgroup of $H=Sp_{4nk}$. It follows that integral \eqref{dou2} is Eulerian, and then it represents the standard $L$-function for the pair $(\pi,\tau)$, after normalizing the Eisenstein series.

The main advantage of the doubling construction is that it applies to any irreducible, automorphic, cuspidal representation of $G({\bf A})$. In other words, it is what we refer to as model free, meaning that for integral \eqref{dou2} to be nonzero (as data vary), one need not assume that $\pi$ has a certain nonzero Fourier coefficient, or any other type of model. 

In  the last 30 years, people found many Rankin-Selberg integrals which represent the above $L$- functions, given by the (generalized) doubling method. Most of them depend on a certain model. For example, for $G=Sp_{2k}$, a global integral for the tensor product $L$-function $L(\pi\times\tau,s)$ was introduced in \cite{G-R-S1}. This construction assumes that an appropriate Whittaker coefficient of the cuspidal representation $\pi$ is not zero. Another construction, this time using a Fourier-Jacobi model, also referred to as a Fourier-Jacobi mixed model, was introduced in \cite{G-J-R-S}. These two constructions use a model which is proved to be unique. This property guarantees that the integrals are indeed Eulerian. There are also examples of integrals obtained by the so-called New Way method. A basic example of such  integrals is introduced in \cite{PS-R2}. They represent the standard $L$-function attached to the representation $\pi$. Here the integrals unfold to a certain model which is not unique. In spite of that, the global integrals are Eulerian. Many other examples of this type of integrals were constructed. See for example \cite{B-F-G}.

In this context, it is only natural to ask  the following questions. First, why do so many different examples of global integrals, which represent the same $L$-function, exist? Second, how does one find these global integrals which turn out to be Eulerian? This issue is especially important when we deal with New Way global integrals. The existence of such integrals is still a mystery. 

In this note we suggest an answer to these questions. Our assertion is that all the integrals which represent the above $L$-functions (represented by the generalized doubling method) can be derived by a relatively  simple procedure from one basic identity. This includes the above constructions of global integrals which use  a unique model of any type, as well as the New Way integrals. In fact, our derivation explains why these types of integrals show up. Moreover, we predict that all integrals which represent these $L$-functions can be derived in the way we shall now explain. We do this for $(G,H,\pi, \tau)$ in the set up of \cite{C-F-G-K1}, where $G$ is a split classical group.

Starting with integral \eqref{dou2}, we define
\begin{equation}\label{dou3}
\xi(\varphi_\pi,f_{\mathcal{E}_\tau,s})(g)=
\int\limits_{G(F)\backslash G({\bf A})}
\int\limits_{U(F)\backslash U({\bf A})}
\varphi_\pi(h)E(f_{\mathcal{E}_\tau,s})(ut(g,h),s)\psi^{-1}_U(u)dudh
\end{equation}
Here, $\varphi_\pi$ is a cusp form in the space of $\pi$. In \eqref{dou2}, we took the complex conjugate of $\pi$ instead of $\pi$, so that $\varphi_\pi$ is in place of $\bar{\varphi}_2$.
We have
\begin{theorem}\label{th1}
Let $f_{\mathcal{E}_\tau,s}$ be a $K$-finite, holomorphic section. Then the function $\xi(\varphi_\pi,f_{\mathcal{E}_\tau,s})$ is meromorphic and takes values in a certain outer conjugate of $\pi$ by an element of order 2, $\pi^\iota$. Moreover, the right hand side of equation \eqref{dou3} is Eulerian, and represents the same partial $L$-function that integral \eqref{dou2} represents, after normalizing the Eisenstein series.
\end{theorem}
The theorem follows from the Euler product expansion of integral \eqref{dou2}, and we will prove it in Sec. \ref{not3}. A similar theorem with the same proof holds in the set up of \cite{PS-R1}, \cite{G-H}. 

We claim that all known constructions of global integrals, which represent the above  $L$- functions, are derived from identity \eqref{dou3}. The way to derive them is as follows. Suppose we are given a global integral which unfolds to a certain global model of $\pi$, via a Fourier coefficient, or a period integral etc. This model may or may not be unique. Then, to derive this global integral from identity \eqref{dou3}, we apply this model to $\xi(\varphi_{\pi^\iota},f_{\mathcal{E}_\tau,s})$, thinking of it as a cusp form in $\pi$. We compute it by means of Fourier expansions and the use of identities relating Eisenstein series as in \cite{G-S}. Then at the end of this computation we obtain, as an inner integration, the same global integral which we started with. This is an important point, and we would like to emphasize it. Starting with the global model above on the function $\xi(\varphi_{\pi^\iota},f_{\mathcal{E}_\tau,s})$, we can always first  unfold the Eisenstein series which appears in identity \eqref{dou3}. However, this is {\sl not} what we do. Instead, by performing some Fourier expansions, and using the identities relating Eisenstein series as in \cite{G-S}, we can derive a "simpler" global integral, which we conjecture to be Eulerian and represent the above $L$-function.

Notice that this procedure implies that for every given global model defined on the space of $\pi$, one can construct an Eulerian integral which unfolds to this model. 
We also predict that all future constructions of global integrals which  represent the above $L$-functions, will also be derived from  identity \eqref{dou3} by a similar procedure. 

Clearly, all the above refers to those cases for which an identity similar to identity \eqref{dou3}, satisfying Theorem \ref{th1}, exists. There are many other global integral constructions which unfold to $L$- functions not represented by the doubling method.

In this paper, we give several examples of the procedure described above. Rather than concentrate on one general example, we decided to give several low rank examples. We believe that by doing so it will be easier for the reader to see the general idea. 

Our first example is the famous Jacquet-Langlands integral construction for the standard $L$-function attached to an irreducible, automorphic, cuspidal representation of $GL_2(\bf A)$. See \cite{J-L}. In Section \ref{gen}, we show how to derive the Jacquet-Langlands integral from the doubling integral of Piatetski-Shapiro and Rallis in \cite{PS-R1}, representing the same $L$-function. 

In Section \ref{new}, we give two examples of  global  integrals involving New Way type integrals. The first is the famous integral of Piatetski-Shapiro and Rallis in \cite{PS-R2}. It represents the standard $L$-function for an irreducible, automorphic, cuspidal representation $\pi$ of $Sp_{2k}(\bf A)$. This example we do in general. In the second example, we show how can one use our procedure to construct new global integrals. Here we extend the work of \cite{PS-R2} and construct a New Way type integral which unfolds to the same Fourier coefficient used in \cite{PS-R2}. We then conjecture that this integral is Eulerian and represents the tensor product $L(\pi\times\tau,s)$, where $\tau$ is an irreducible, automorphic, cuspidal representation of $GL_n(\bf A)$.

In the last section we give the example of the global integral construction of \cite{G-J-R-S} for $\pi$ on the double cover of $Sp_4(\bf A)$ and $\tau$ on $GL_2(\bf A)$, involving a Fourier-Jacobi mixed model of $\pi$ with respect to an irreducible, automorphic, cuspidal representation of $SL_2(\bf A)$. We explain through our procedure how does this integral show up.

\section{Notations and Preliminarries}\label{notation}
\subsection{Some General Notations}\label{not1}
We start with some general notations. 
Since most of the examples in this note are for the symplectic group, we fix some notations for this specific group. 

For a positive integer $k$, let $J_{k}$ denote the matrix of size $k$, which has ones on the  anti-diagonal and zeros elsewhere. We realize $Sp_{2k}$ as 
\begin{equation}\label{sp1}\notag
Sp_{2k}=\{ g\in GL_{2k}\ :\ g^t\begin{pmatrix} &J_k\\ -J_k&\end{pmatrix} g=
\begin{pmatrix} &J_k\\ -J_k&\end{pmatrix}\}.
\end{equation}
Denote 
$$
Mat_k^0=\{A\in Mat_k\ :\ A^tJ_k=J_kA\}. 
$$
We denote by $Q_k$ the (standard) Siegel parabolic subgroup of $Sp_{2k}$. More generally, for a partition $k_1,...,k_{r+1}$ of $k$, we denote by $Q_{k_1,...,k_r,2k}$ the standard parabolic subgroup of $Sp_{2k}$, whose Levi part is isomorphic to $GL_{k_1}\times\cdots\times GL_{k_r}\times Sp_{2k_{r+1}}$. We denote its unipotent radical by $U_{k_1,...,k_r,2k}$.

For two positive integers $n_1$ and $n_2$,  denote by $0_{n_1\times n_2}$ the zero matrix in $Mat_{n_1\times n_2}$. We shall write $0_n$ for $0_{1\times n}$. Let $e_{i,j}$ denote the matrix of size $2k$, which has one at the $(i,j)$ entry, and zero elsewhere. Let 
$$
e_{i,j}'=e_{i,j}-e_{2k-j+1,2k-i+1},\quad  if \quad 1\le i,j\le k, 
$$
and
$$  
e_{i,j}'=e_{i,j}+e_{2k-j+1,2k-i+1},\quad  if \quad 1\le i\le k,\quad j>k.
$$

Let $F$ be a number field and let ${\bf A}$ denote its ring of Adeles. Fix $\psi$, a non-trivial character of $F\backslash {\bf A}$. Let $Sp_{2k}^{(2)}({\bf A})$ denote the metaplectic double cover of $Sp_{2k}({\bf A})$. We denote a theta series on $Sp^{(2)}_{2k}({\bf A})$, corresponding to $\psi$ and a Schwartz function $\phi\in \mathcal{S}({\bf A}^k)$, by $\theta_{\psi,Sp_{2k}}^\phi$, or shortly $\theta_{\psi,2k}^\phi$. We will rely on the well-known properties of this series. For notations and the action of the Weil representation we  use the formulas as given in \cite{G-R-S1}, page 188.

\subsection{Eisenstein Series}\label{not2}
In this sub-section we consider certain Eisenstein series which we will use later. We also state some relevant identities which they satisfy. All the results in this sub-section are proved in \cite{G-S}, or follow from it. 

For an irreducible, automorphic, cuspidal representation $\tau$ of $GL_2({\bf A})$, and a natural number $m$, let $\Delta(\tau,m)$ denote the Speh representation of $GL_{2m}({\bf A})$. This representation was studied in \cite{J}. Let $Q_{2m}$ denote the Siegel parabolic subgroup of $Sp_{4m}$.  Consider an Eisenstein series $E(f_{\Delta(\tau,m),s})$ on $Sp_{4m}({\bf A})$, corresponding to a smooth holomorphic section $f_{\Delta(\tau,m),s}$ of $Ind_{Q_{2m}({\bf A})}^{Sp_{4m}({\bf A})}\Delta(\tau,m)|\text{det}\cdot|^s$ (normalized induction). See \cite{G-S}, Sec. 2. Sometimes it will be convenient to drop the notation of the section and simply denote $E_{\tau,m}(\cdot,s)$. 

We will also need similar Eisenstein series $E^{(2)}(f_{\Delta(\tau,m)\gamma_\psi,s})$ on  $Sp_{4m}^{(2)}({\bf A})$. It corresponds to a smooth, holomorphic section of 
$Ind_{Q_{2m}^{(2)}({\bf A})}^{Sp_{4m}^{(2)}({\bf A})}\Delta(\tau,m)
\gamma_\psi|\text{det}\cdot|^s$. Here $\gamma_\psi$ is the Weil factor attached to the character $\psi$. See \cite{G-S}, Sec. 2. Again, sometimes it will be convenient to simply use the notation $E^{(2)}_{\tau,m}(\cdot,s)$. 

Consider the unipotent radical $U_{1^2,4m}$. Let $\psi_1$ denote the following character of $U_{1^2,4m}(\bf A)$ . For $u=(u_{i,j})\in U_{1^2,4m}$,  
$$
\psi_1(u)=\psi(u_{1,2}). 
$$
Let $U_{1^2,4m}^0$ denote the subgroup of $U_{1^2,4m}$ consisting of all matrices 
$u=(u_{i,j})\in U_{1^2,4m}$, such that $u_{2,j}=0$ for all $3\le j\le 4m-2$. For $t\in F^*$, let $\psi_{U_{1^2,4m}^0,t}$ denote the following character of $U_{1^2,4m}^0(\bf A)$, 
$$
\psi_{U_{1^2,4m}^0,t}(u)=\psi(u_{1,2}+tu_{2,4m-1}).
$$
 There is a natural projection $j=j_{4m-3}$, from $U_{1^2,4m}$ onto the Heisenberg group in $4m-3$ variables ${\mathcal H}_{4m-3}$. In coordinates, given $u=(u_{i,j})\in U_{1^2,4m}$, 
$$
j(u)=(u_{2,3},u_{2,4},\ldots,u_{2,4m-2},u_{2,4m-1}). 
$$
Here, for elements in ${\mathcal H}_{4m-3},$ we use the notations given in \cite{G-R-S1}, Sec. 1.
Of course, we have a similar projection from $U_{1,2m}$ onto $\mathcal{H}_{2m-1}$. It is an isomorphism. We denote the inverse map by $i_{2m-1}$.

From \cite{G-S} Theorems 7.1, 8.1, for all $\tilde{h}=(h,\epsilon)\in Sp_{4(m-1)}^{(2)}({\bf A})$, we have 
\begin{equation}\label{id1}
E_{\tau,m-1}^{(2)}(\tilde{h},s)=
\int\limits_{U_{1^2,4m}(F)\backslash U_{1^2,4m}({\bf A})}
\theta_{\psi^{-1},4(m-1)}^\phi(j(u)\tilde{h})
E_{\tau,m}(ut(h),s)\psi_1^{-1}(u)du.
\end{equation}
In a similar way, 
\begin{equation}\label{id2}
E_{\tau,m-1}(h,s)=
\int\limits_{U_{1^2,4m}(F)\backslash U_{1^2,4m}({\bf A})}
\theta_{\psi^{-1},4(m-1)}^\phi(j(u)\tilde{h})
E_{\tau,m}^{(2)}(ut(\tilde{h}),s)\psi_1^{-1}(u)du.
\end{equation}
Here, $t(h)=diag(I_2,h,I_2)$ and $t(\tilde{h})=(t(h),\epsilon)$. We used the shorthand notation for Eisenstein series in order to stress the fact that the r.h.s. of \eqref{id1}(resp. \eqref{id2}) is equal to an Eisenstein series on $Sp_{4(m-1)}^{(2)}({\bf A})$ (resp. $Sp_{4(m-1)}({\bf A})$), namely the Eisenstein series indicated on the l.h.s. of \eqref{id1}(resp. \eqref{id2}). A more precise statement of \eqref{id1} is that given a smooth, holomorphic section $f_{\Delta(\tau,m),s}$ of $Ind_{Q_{2m}({\bf A})}^{Sp_{4m}({\bf A})}\Delta(\tau,m)|\text{det}\cdot|^s$, and a Schwartz function $\phi\in \mathcal{S}({\bf A}^{2m-2})$, there is a smooth, meromorphic section $\Lambda(f_{\Delta(\tau,m),s},\phi)$ of $Ind_{Q_{2m-2}^{(2)}({\bf A})}^{Sp_{4m-4}^{(2)}({\bf A})}\Delta(\tau,m-1)\gamma_{\psi^{-1}}|\text{det}\cdot|^s$, such that, for all $\tilde{h}=(h,\epsilon)\in Sp_{4(m-1)}^{(2)}({\bf A})$, we have  
$$
E^{(2)}(\Lambda(f_{\Delta(\tau,m),s}))(\tilde{h})=\int\limits_{U_{1^2,4m}(F)\backslash U_{1^2,4m}({\bf A})}
\theta_{\psi^{-1},4(m-1)}^\phi(j(u)\tilde{h})
E(f_{\Delta(\tau,m),s})(ut(h),s)\psi_1^{-1}(u)du.
$$
The section $\Lambda(f_{\Delta(\tau,m),s},\phi)$ is given, for $Re(s)$ sufficiently large by an explicit Adelic unipotent integration, and continues to a meromorphic function. A similar theorem holds for \eqref{id2}. Moreover, in \cite{G-S} the above identities are proved to hold for normalized Eisenstein series, as well. 

We will use another identity relating Eisenstein series. We first fix some more notations. Let $r$ and $k$ denote two positive integers such that $k\ge 3r$. Consider the unipotent radical  $U_{2r,2k}$. It consists of all matrices of the form
\begin{equation}\label{id3}
u=\begin{pmatrix} I_{2r}&x_1&y&x_2&z\\ &I_r&&&\star\\  &&I_{2(k-3r)}&&\star\\
&&&I_r&\star\\ &&&&I_{2r}\end{pmatrix}\in Sp_{2k}. 
\end{equation}
Define a character $\psi_{U_{2r,2k}}$ of $U_{2r,2k}(\bf A)$, as follows.  Let $u\in U_{2r,2k}(\bf A)$ of the form \eqref{id3}. For $i=1,2$, write $x_i=\begin{pmatrix} x_{i,1}\\ x_{i,2}\end{pmatrix}$, where $x_{i,j}\in Mat_r(\bf A)$. Then 
$$
\psi_{U_{2r,2k}}(u)=\psi(\text{tr}(x_{1,1}+x_{2,2})).
$$
This is a special case of the character defined in \cite{G-S} in equation (1.2). The stabilizer of $\psi_{U_{2r,2k}}$ inside the Levi subgroup $GL_{2r}({\bf A})\times Sp_{2(k-2r)}(\bf A)$ is the group  $Sp_{2r}({\bf A})\times Sp_{2(k-3r)}(\bf A)$, embedded in $Sp_{2k}(\bf A)$ as 
\begin{equation}\label{id4}
t(g,h)= \text{diag}(g,j(g,h),g^*), \ \ j(g,h)=\begin{pmatrix}
g_1&&g_2\\ &h&\\ g_3&&g_4\end{pmatrix}, \ \ g=\begin{pmatrix} g_1&g_2\\ g_3&g_4 \end{pmatrix}\in Sp_{2r}(\bf A).
\end{equation}

To state the third identity that we need, let $\sigma$ denote an irreducible, automorphic, cuspidal representation of $SL_2({\bf A})$, and let $r=1$, $k=6$. For $\tau$ as above, form the Eisenstein series $E_{\tau,\sigma}(\cdot,s)$ on $Sp_6({\bf A})$ associated with $Ind_{Q_{2^2,6}({\bf A})}^{Sp_6({\bf A})}(\tau\times\sigma)|\text{det}\cdot|^s$. Then, it follows from \cite{G-S}, Theorem 1.1 that, for any cusp form $\varphi_\sigma$ in the space of $\sigma$,
\begin{equation}\label{id5}
E_{\tau,\sigma}(h,s)=
\int\limits_{SL_2(F)\backslash SL_2({\bf A})}
\int\limits_{U_{2,12}(F)\backslash U_{2,12}({\bf A})}
\varphi_\sigma(g^\iota)E_{\tau,3}(ut(g,h),s)\psi_{U_{2,12}}^{-1}(u)dudg.
\end{equation}
Here, for a matrix $h\in SL_2$, $h^\iota=J_2hJ_2$, where we recall that $J_2=\begin{pmatrix} &1\\ 1&\end{pmatrix}$. Again, the identity \eqref{id5} is written using a shorthand notation, dropping mention of the sections. The point is that the r.h.s. of \eqref{id5} is equal to an Eisenstein series on $Sp_6(\bf A)$, namely an Eisenstein series corresponding to the parabolic induction $Ind_{Q_{2^2,6}({\bf A})}^{Sp_6({\bf A})}(\tau\times\sigma)|\text{det}\cdot|^s$. Note the form of the r.h.s. of \eqref{id5}. It is a kernel integral, where the kernel function on $SL_2({\bf A})\times Sp_6(\bf A)$ is the $\psi_{U_{2,12}}$-Fourier coefficient of the Eisenstein series $E_{\tau,3}$ on $Sp_{12}(\bf A)$. Then we integrate this kernel function against a cusp form on $SL_2(\bf A)$ to obtain an Eisenstein series $E_{\tau,\sigma}(\cdot,s)$ on $Sp_6(\bf A)$. The precise statement of \eqref{id5} says that for a given smooth, holomorphic section $f_{\Delta(\tau,3),s}$ of $Ind_{Q_6({\bf A})}^{Sp_{12}({\bf A})}\Delta(\tau,3)|\text{det}\cdot|^s$, and a cusp form $\varphi_\sigma$ in the space of $\sigma$, there is a smooth, meromorphic section $\Lambda(f_{\Delta(\tau,3),s},\varphi_\sigma)$ of $Ind_{Q_{2,2}({\bf A})}^{Sp_6({\bf A})}(\tau\times\sigma)|\text{det}\cdot|^s$, such that, for all $h\in Sp_6(\bf A)$,
$$
E(\Lambda(f_{\Delta(\tau,3),s},\varphi_\sigma))(h)=
\int\limits_{SL_2(F)\backslash SL_2({\bf A})}
\int\limits_{U_{2,12}(F)\backslash U_{2,12}({\bf A})}
\varphi_\sigma(g^\iota)E(f_{\Delta(\tau,3),s})(ut(g,h),s)\psi_{U_{2,12}}^{-1}(u)dudg.
$$
The section $\Lambda(f_{\Delta(\tau,3),s},\varphi_\sigma)$ is given, for $Re(s)$ sufficiently large, by an explicit Adelic integration, and continues to a meromorphic function. Moreover, in \cite{G-S} the above identity is proved to hold for normalized Eisenstein series, as well.

\subsection{Proof of Theorem \ref{th1}}\label{not3}
We prove Theorem \ref{th1} for the set up of \cite{C-F-G-K1}. For simplicity of notation we prove the theorem for $G=Sp_{2k}$. The proof for the other split classical groups is almost entirely the same, and similarly for the set up of \cite{PS-R1}, where the classical group need not be split. The proof for the case in \cite{G-H} is also quite similar.

We specify some of the data in \eqref{dou2} and \eqref{dou4}. The group $H$ is $Sp_{4nk}$, and $U=U_{(2k)^{n-1},4nk}$. Recall that this is the unipotent radical of the standard parabolic subgroup of $H$, with Levi part isomorphic to $GL_{2k}^{n-1}\times Sp_{4k}$. The character $\psi_U$ is stabilized by $G({\bf A})\times G(\bf A)$, embedded in $H(\bf A)$ by
$t(g,h)=\text{diag}(g^{\Delta_{n-1}},j(g,h),(g^*)^{\Delta_{n-1}})$, where $g^{\Delta_{n-1}}=\text{diag}(g,...,g)$, $n-1$ times, and, for $g,h\in Sp_{2k}(\bf A)$, $j(g,h)$ has the same shape as \eqref{id4}. The representation $\mathcal{E}_\tau$ is the Speh representation $\Delta(\tau,2k)$, and $W(\Delta(\tau,2k))$ denotes the model referred to as a Whittaker-Speh-Shalika model in \cite{C-F-G-K1}. It is obtained by applying a Fourier coefficient on $\Delta(\tau,2k)$, corresponding to $\psi$ and the partition $(n^{2k})$. Next, the element $\delta$ in \eqref{dou4} is
$$
\delta=\begin{pmatrix}0&I_{2k}&0&0\\0&0&0&I_{2k(n-1)}\\-I_{2k(n-1)}&0&0&0\\0&I_{2k}&I_{2k}&0\end{pmatrix}.
$$
Finally, the unipotent group $U_0$ is a subgroup of $U_{(2k)^{n-1},4nk}$, realizing the quotient
$\delta^{-1}Q_{2nk}\delta\cap U_{(2k)^{n-1},4nk}\backslash U_{(2k)^{n-1},4nk}$.

Denote, for  $b\in G({\bf A})$, $b^\iota=\begin{pmatrix}&I_k\\I_k\end{pmatrix}b\begin{pmatrix}&I_k\\I_k\end{pmatrix}$.

The proof is a simple consequence of the unfolding process of the global integrals \eqref{dou2} and \eqref{dou3}. A careful inspection of the unfolding process implies the following two facts. 

First, as a function of $g,h\in G({\bf A})$, the integral 
\begin{equation}\label{dou41}
\int\limits_{ U_0({\bf A})}
f_{W({\mathcal E}_\tau),s}(\delta u_0t(g,h))\psi^{-1}_U(u_0)du_0
\end{equation}
is left invariant by the following (almost) diagonal embedding of $G({\bf A})$. For all $b\in G({\bf A})$,
\begin{equation}\label{dou5}
\int\limits_{ U_0({\bf A})}
f_{W({\mathcal E}_\tau),s}(\delta u_0t(bg,b^\iota h))\psi^{-1}_U(u_0)du_0=
\int\limits_{ U_0({\bf A})}
f_{W({\mathcal E}_\tau),s}(\delta u_0t(g,h))\psi^{-1}_U(u_0)du_0.
\end{equation}
The second fact is that carrying out the unfolding process for the integral on the right hand side of \eqref{dou3}, we obtain for $\text{Re}(s)$ large, 
\begin{equation}\label{dou6}
\xi(\varphi_\pi,f_{\mathcal{E}_\tau,s})(g)=
\int\limits_{G({\bf A})}
\int\limits_{U_0({\bf A})}
\varphi_\pi(h)f_{W({\mathcal E}_\tau),s}(\delta u_0t(g,h))\psi^{-1}_U(u_0)du_0dh.
\end{equation}

A similar calculation was also done in \cite{G-S} Sections 2 an 3, and a similar identity is given in  \cite{G-S} equation (3.26). 

Using \eqref{dou5} in \eqref{dou6} and a simple change of variables, we get for $\text{Re}(s)$ large, 
\begin{equation}\label{dou6.1}
\xi(\varphi_\pi,f_{\mathcal{E}_\tau,s})(g)=
\int\limits_{G({\bf A})}
\int\limits_{U_0({\bf A})}
\varphi_\pi(g^\iota h)f_{W({\mathcal E}_\tau),s}(\delta u_0t(1,h))\psi^{-1}_U(u_0)du_0dh.
\end{equation}
Hence, for $\text{Re}(s)$ large, 
\begin{equation}\label{dou6.2}
\xi(\varphi_\pi,f_{\mathcal{E}_\tau,s})=
\int\limits_{G({\bf A})}(
\int\limits_{U_0({\bf A})}
f_{W({\mathcal E}_\tau),s}(\delta u_0t(1,h))\psi^{-1}_U(u_0)du_0)\iota(\pi(h)\varphi_\pi) dh,
\end{equation}
where for a cusp form $\varphi_\pi$, $\iota(\varphi_\pi)(g)=\varphi_\pi(g^\iota)$. Fix an isomorphism $\ell: \otimes_v \pi_v\mapsto \pi$, and assume that $\varphi_\pi$ is the image of a decomposable vector $\otimes_v\varphi_{\pi_v}$. Assume, also, that $f_{W({\mathcal E}_\tau),s}$ is a product of local sections $f_{W({\mathcal E}_{\tau_v}),s}$, where $W({\mathcal E}_{\tau_v})$ denotes the corresponding local unique model of $\mathcal{E}_{\tau_v}$. Let $S$ be a finite set of places, containing the Archimedean places, and outside which $\pi$, $\tau$ and $\psi$ are unramified, and for $v\notin S$, $\varphi_{\pi_v}=\varphi^0_{\pi_v}$ is a pre-chosen unramified vector, and $f_{W({\mathcal E}_{\tau_v}),s}=f^0_{W({\mathcal E}_{\tau_v}),s}$ is unramified and normalized, such that its value at the identity is the unique unramified element in $W({\mathcal E}_{\tau_v})$, which has the value 1 at $I$. For $\text{Re}(s)$ large, 
\begin{equation}\label{dou6.3}
\xi(\varphi_\pi,f_{\mathcal{E}_\tau,s})=
(\iota\circ\ell)(\otimes_v\int\limits_{G(F_v)}(
\int\limits_{U_0(F_v)}
f_{W({\mathcal E}_{\tau_v}),s}(\delta u_0t(1,h))\psi^{-1}_{v,U}(u_0)du_0)\pi_v(h)\varphi_{\pi_v} dh).
\end{equation}
Denote the factor at $v$, inside $\iota\circ\ell$, in \eqref{dou6.3}, by $\xi(\varphi_{\pi_v},f_{W({\mathcal E}_{\tau_v}),s})$. Let us compute this factor at $v\notin S$. Let $K_{G,v}$, $K_{H,v}$ denote the standard maximal compact subgroups of $G(F_v)$ and $H(F_v)$ respectively. For $r\in K_{G,v}$, change variable $h\mapsto rh$ and integrate over $r\in K_{G,v}$, taking the measure of $K_{G,v}$ to be 1. We have 
$$
f^0_{W({\mathcal E}_{\tau_v}),s}(\delta u_0t(1,rh))=f^0_{W({\mathcal E}_{\tau_v}),s}(\delta u_0t(r^\iota,rh)).
$$
Using \eqref{dou5},
\begin{equation}\label{dou6.4}
\xi(\varphi^0_{\pi_v},f^0_{W({\mathcal E}_{\tau_v}),s})=
\int\limits_{G(F_v)}(
\int\limits_{U_0(F_v)}
f^0_{W({\mathcal E}_{\tau_v}),s}(\delta u_0t(1,h))\psi^{-1}_{v,U}(u_0)du_0)\int\limits_{K_{G,v}}\pi_v(rh)\varphi^0_{\pi_v}dr dh.
\end{equation}
The $dr$-integration in \eqref{dou6.4} is equal to $\omega_{\pi_v}(h)\varphi^0_{\pi_v}$, where $\omega_{\pi_v}$ is the unique spherical function attached to $\pi_v$. Hence, for $v\notin S$, and $\text{Re}(s)$ large, 
\begin{equation}\label{dou6.5}
\xi(\varphi^0_{\pi_v},f^0_{W({\mathcal E}_{\tau_v}),s})=
(\int\limits_{G(F_v)}
\int\limits_{U_0(F_v)}
f^0_{W({\mathcal E}_{\tau_v}),s}(\delta u_0t(1,h))\psi^{-1}_{v,U}(u_0)\omega_{\pi_v}(h)du_0dh)\varphi^0_{\pi_v}.
\end{equation}
The integral in \eqref{dou6.5} is equal to $\frac{L(\pi_v\times\tau_v,s+\frac{1}{2})}{d_v(s)}$, where the denominator comes from the normalizing factor of the Eisenstein series in the global integral \eqref{dou2}. This is the unramified computation carried in \cite{C-F-G-K1} (and similarly in \cite{PS-R1}, \cite{G-H}.) Denote $\varphi_\pi^S=\otimes_{v\notin S}\varphi^0_{\pi_v}$. Then it follows that  
\begin{equation}\label{dou6.6}
\xi(\varphi_\pi,f_{\mathcal{E}_\tau,s})=\frac{L^S(\pi\times\tau,s+\frac{1}{2})}{d^S(s)}
(\iota\circ\ell)(\otimes_{v\in S}\xi(\varphi_{\pi_v},f_{W({\mathcal E}_{\tau_v}),s})\otimes \varphi_\pi^S).
\end{equation}
For $v\in S$, a similar argument, using the $K_{H,v}$-finiteness of $f_{W({\mathcal E}_{\tau_v}),s}$, shows that $\xi(\varphi_{\pi_v},f_{W({\mathcal E}_{\tau_v}),s})$ is equal to a finite sum of elements of the form
\begin{equation}\label{dou6.7}
(\int\limits_{G(F_v)}
\int\limits_{U_0(F_v)}
f'_{W({\mathcal E}_{\tau_v}),s}(\delta u_0t(1,h))\psi^{-1}_{v,U}(u_0)<\pi_v(h)\varphi_v,\alpha_{\pi_v}>du_0dh)\beta_{\pi_v},
\end{equation}
where $<\pi_v(h)\varphi_{\pi_v},\alpha_v>$ denotes a matrix coefficient of $\pi_v$ corresponding to $\varphi_{\pi_v}$ and a vector $\alpha_v$ in the dual of $\pi_v$; $\beta_{\pi_v}$ is a vector in the space of $\pi_v$. Denote the integral in \eqref{dou6.7} by $\mathcal{L}(\varphi_{\pi_v},\alpha_{\pi_v},f'_{W({\mathcal E}_{\tau_v}),s})$. This integral is absolutely convergent for $\text{Re}(s)$ large, and continues to a meromorphic function in the complex plane. This is the local integral of the doubling method. See \cite{C-F-G-K1}. All in all, we can express the meromorphic function $\xi(\varphi_\pi,f_{\mathcal{E}_\tau,s})$ as a finite sum of elements of the form
\begin{equation}\label{dou6.8}
\frac{L^S(\pi\times\tau,s+\frac{1}{2})}{d^S(s)}
\prod_{v\in S}\mathcal{L}(\varphi_{\pi_v},\alpha_{\pi_v},f'_{W({\mathcal E}_{\tau_v}),s})\cdot (\iota\circ\ell)(\otimes_{v\in S}\beta_{\pi_v}\otimes \varphi_\pi^S)^\iota.
\end{equation}
In particular, $\xi(\varphi_\pi,f_{\mathcal{E}_\tau,s})$ lies in the space of $\pi^\iota$. This completes the proof of Theorem \ref{th1}.

\subsection{The Lemma on Exchanging Roots}\label{not4}
We will need the lemma on exchanging roots, Lemma 7.1 in \cite{G-R-S2}. Typically, we have $F$-unipotent subgroups $A, B, C, D, X, Y$, of $H$, such that $X, Y$ are abelian, intersect $C$ trivially, normalize $C$, and satisfy $[X,Y]\subset C$, $B=CY$, $D=CX$, $A=BX=DY$. We have a nontrivial character $\psi_C$ of $C(\bf A)$, trivial on $C(F)$, such that $X({\bf A}), Y(\bf A)$ preserve $\psi_C$, when acting by conjugation, and, finally, we assume that the form $(x,y)\mapsto \psi_C([x,y])$ defines a non-degenerate bi-multiplicative pairing on $X({\bf A})\times Y(\bf A)$. Extend the character $\psi_C$ to a character $\psi_B$ of $B(\bf A)$ and a character $\psi_D$ of $D(\bf A)$, by making $\psi_B$ trivial on $Y(\bf A)$ and $\psi_D$ trivial on $X(\bf A)$.
\begin{lemma}\label{lemmix2}
1. Let $f$ be a smooth, automorphic function on $H(\bf A)$, of moderate growth. Then for every  $h\in H(\bf A)$,
\begin{equation}\label{mix10}
\int\limits_{B(F)\backslash B({\bf A})}
f(vh)\psi^{-1}_B(v)dv=
\int\limits_{Y({\bf A})}
\int\limits_{D(F)\backslash D({\bf A})}
f(uyh)\psi^{-1}_D(u)dudy.
\end{equation}
2. Let $\mathcal A$ be a space of smooth automorphic functions of moderate growth on $H(\bf A)$. Assume that $\mathcal A$ is invariant to right translations by elements of $H(\bf A)$, and that, as an $H(\bf A)$-module, it satisfies the Dixmier-Malliavin Lemma (\cite{D-M}). Let $\Omega\subset H(\bf A)$ be the subset of elements $h$ which satisfy the following two properties. 
\begin{enumerate}
	\item $[Y({\bf A}),h]\subset D$ and $\psi_D([Y({\bf A}),h])=1$;
	\item For every $x\in X(\bf A)$, $y\in Y(\bf A)$, $[yxy^{-1},h]\in D(\bf A)$ and
	$\psi_D([yxy^{-1},h])=1$.	
\end{enumerate}
Then for each $f\in \mathcal A$, there exists  $\varphi\in {\mathcal A}$ such that for all $h\in \Omega$,
\begin{equation}\label{mix11}
\int\limits_{B(F)\backslash B({\bf A})}
f(vh)\psi_B^{-1}(v)dv=
\int\limits_{D(F)\backslash D({\bf A})}
\varphi(uh)\psi^{-1}_D(u)du.
\end{equation}
3. With the same assumptions as in the beginning of (2), let $\Omega'\subset H({\bf A})$ be the subset of elements $h$ which satisfy the following three properties. 
\begin{enumerate}
	\item $h$ normalizes $Y({\bf A})$ and preserves the measure $dy$;
	\item $h$ normalizes $X({\bf A})$;
	\item For every $x\in X(\bf A)$, $y\in Y(\bf A)$, $[h,[y,x]]\in D(\bf A)$ and
	$\psi_D([h,[y,x]])=1$.	
\end{enumerate}
Then the assertion of (2) holds for $h\in \Omega'$. 
\end{lemma}

\begin{proof}
The first statement \eqref{mix10} is Lemma 7.1 in \cite{G-R-S2}. The proof of the second statement follows by similar arguments as in \cite{G-R-S2}, Corollary 7.2. Let $f\in \mathcal A$. By the lemma of Dixmier and Malliavin, there exist $f_1,...,f_r\in \mathcal A$, and $\xi_1,...,\xi_r\in C_c^\infty(X({\bf A}))$, such that for all $z\in H(\bf A)$,
\begin{equation}\label{mix12}
f(z)=\sum_{i=1}^r\int\limits_{X({\bf A})}\xi_i(x)f_i(zx)dx.
\end{equation}
Plug this into the right hand side of equation \eqref{mix10}. We obtain
\begin{equation}\label{mix13}
\sum_{i=1}^r\
\int\limits_{Y({\bf A})} 
\int\limits_{D(F)\backslash D({\bf A})}
\int\limits_{X({\bf A})}
\xi_i(x)f_i(uyhx)
\psi^{-1}_D(u)dxdudy.
\end{equation}
Since $h\in \Omega$, the first property defining $\Omega$ shows that, for $y\in Y(\bf A)$, there is $u'\in D(\bf A)$, such that $yh=u'hy$ and $\psi_D(u')=1$. Changing variable, $uu'\mapsto u$, we get that in the integral \eqref{mix13}, we can replace $yh$ by $hy$, and get
\begin{equation}\label{mix14}
\sum_{i=1}^r\ 
\int\limits_{Y({\bf A})} 
\int\limits_{D(F)\backslash D({\bf A})}
\int\limits_{X({\bf A})}
\xi_i(x)f_i(uhyx)
\psi^{-1}_D(u)dxdudy.
\end{equation}
By the second condition defining $\Omega$, we have in \eqref{mix14}, $hyx=[y,x]u'hy$, where $u'\in D(\bf A)$ is such that $\psi_D(u')=1$. Change variable $u[y,x]u'\mapsto u$, to get 
\begin{equation}\label{mix16}
\sum_{i=1}^r\ 
\int\limits_{D(F)\backslash D({\bf A})}
\int\limits_{Y({\bf A})}
\hat{\xi}_i(y)f_i(uhy)
\psi^{-1}_D(u)dydu,
\end{equation}
where
\begin{equation}\label{mix15}\notag
\hat{\xi}_i(y)=\int\limits_{X({\bf A})}\xi_i(x)\psi^{-1}_D([x,y])dy.
\end{equation}
Choosing 
\begin{equation}\label{mix17}\notag
\varphi(z)=\sum_{i=1}^r\int\limits_{Y({\bf A})}\hat{\xi}_i(y)f_i(zy)dy,\quad z\in H(\bf A),
\end{equation}
we get \eqref{mix11}. 

The proof of the third part is similar. Here, for $h\in \Omega'$, by the first property defining $\Omega'$, the r.h.s. of \eqref{mix10} is equal to
\begin{equation}\label{mix18}
\int\limits_{Y({\bf A})}
\int\limits_{D(F)\backslash D({\bf A})}
f(uhy)\psi^{-1}_D(u)dudy.
\end{equation}
Write $f$ as in \eqref{mix12}, and now \eqref{mix18} becomes
\begin{equation}\label{mix19}
\sum_{i=1}^r\
\int\limits_{Y({\bf A})} 
\int\limits_{D(F)\backslash D({\bf A})}
\int\limits_{X({\bf A})}
\xi_i(x)f_i(uhyx)
\psi^{-1}_D(u)dxdudy.
\end{equation}
By the third property defining $\Omega'$, \eqref{mix19} is equal to
\begin{equation}\label{mix20}
\sum_{i=1}^r\
\int\limits_{Y({\bf A})}
\int\limits_{X({\bf A})} 
\int\limits_{D(F)\backslash D({\bf A})}
\xi_i(x)f_i(uhxy)\psi^{-1}([x,y])
\psi^{-1}_D(u)dudxdy.
\end{equation}
Since $h$ normalizes $X({\bf A})$, changing variable $uhxh^{-1}\mapsto u$, we get that \eqref{mix20} is equal to \eqref{mix16}.
 
\end{proof}

\section{Example: the Jacquet - Langlands integral}\label{gen}
Let $\pi$ be an irreducible, automorphic, cuspidal representation of $GL_n({\bf A})$. We assume, for simplicity, that $\pi$ has a trivial central character. A doubling integral for the standard $L$-function of $\pi$ is given in \cite{PS-R1}, Section 3. Proposition 3.2 in \cite{PS-R1} relates the resulting doubling integral to the global integral of Godement-Jacquet \cite{G-J}. More precisely, by correctly choosing the section defining the Eisenstein series, this doubling integral represents $L(\pi,ns-\frac{1}{2}(n-1))$. This is the section denoted by $f^\phi(h;s)$ on p. 12 in \cite{PS-R1}. When we take it to be decomposable, its local components almost everywhere are not normalized; the product of their values at $1$, at all places $v$ outside an appropriate finite set $S$ containing the Archimedean places, is $L^S(4s)$. In this case, \eqref{dou3} is given by  
\begin{equation}\label{gen1}
\xi(\varphi_\pi,f_s)(g)=
\int\limits_{Z({\bf A})GL_n(F)\backslash GL_n({\bf A})}
\varphi_\pi(g)E(f_s)(t(g,h))dh.
\end{equation}
The Eisenstein series is defined on the group $GL_{n^2}({\bf A})$, and it is attached to the section $f_s$ of the induced representation $Ind_{P({\bf A})}^{GL_{n^2}({\bf A})}\delta_P^{s-\frac{1}{2}}$ (normalized induction). Here $P$ is the maximal parabolic subgroup of $GL_{n^2}$, whose Levi part is isomorphic to $GL_1\times GL_{n^2-1}$.
The embedding of $(g,h) \in GL_n({\bf A})\times GL_n({\bf A})$, $t(g,h)$, is given by the tensor product map. Also, $Z$ is the center of $GL_n$. 

In the rest of this section we assume that $n=2$. Then the doubling integral mentioned above, and motivating the definition of \eqref{gen1}, represents $\frac{L(\pi, 2s-\frac{1}{2})}{L(4s)}$. (Now we assume that the data in \eqref{gen1} are decomposable and the section $f_s$ is normalized at almost all places.) Recall the well known global integral of Jacquet and Langlands in \cite{J-L}, representing the standard $L$-function $L(\pi,s)$. This is the integral
\begin{equation}\label{gen2}
\int\limits_{F^*\backslash {\bf A}^*}\varphi_\pi\begin{pmatrix} a&\\ &1\end{pmatrix}|a|^{s-1/2}da^*.
\end{equation}
This integral unfolds to an integral which involves the $\psi$- Whittaker coefficient of $\varphi_\pi$. Hence, we will now prove that when computing the $\psi$- Whittaker coefficient of $\xi(\varphi_\pi,f_s)$, we obtain the integral \eqref{gen2}, with $s$ shifted, as inner integration. 

The $\psi$-Whittaker coefficient of $\xi(\varphi_\pi,f_s)$ is the integral
\begin{equation}\label{gen3}
\int\limits_{Z({\bf A})GL_2(F)\backslash GL_2({\bf A})}
\int\limits_{F\backslash {\bf A}}
\varphi_\pi(h)E(f_s)\left (\begin{pmatrix} I_2& xI_2\\ &I_2\end{pmatrix}\begin{pmatrix} h&\\ &h\end{pmatrix},s\right )\psi^{-1}(x)dxdh.
\end{equation}
We prove
\begin{theorem}\label{prop1}
For all $h\in GL_2(\bf A)$ the integral
\begin{equation}\label{gen31}
\int\limits_{F^*\backslash {\bf A}^*}
\varphi_\pi(\begin{pmatrix} a& \\ &1\end{pmatrix}h)|a|^{2s-1}d^*a
\end{equation}
is an inner integration to integral \eqref{gen3}.
\end{theorem}

\begin{proof}

Let $U$ denote the unipotent radical of the standard parabolic subgroup of $GL_4$, whose Levi part is isomorphic to $GL_2\times GL_2$. Thus, $U({\bf A})$ consists of all matrices of the form
\begin{equation}\label{gen4}\notag
u(y)=\begin{pmatrix} I_2&y\\&I_2\end{pmatrix},
\ \ \ \ \ \ y\in Mat_2({\bf A}).
\end{equation}
Let $\psi_U$ be the character of $U(\bf A)$ defined by $\psi_U(u(y))=\psi(y_{2,1})$. Write the Fourier expansion of $E(f_s)$ along $U(F)\backslash U({\bf A})$. Each Fourier coefficient corresponds to a matrix $\gamma\in Mat_2(F)$, which defines the character $\psi_{U,\gamma}(u(y))=\psi(tr(\gamma y))$, and the corresponding Fourier coefficient is
$$
E^{\psi_{U,\gamma}}(f_s)(m)=\int\limits_{U(F)\backslash U({\bf A})}
E(f_s)(um)\psi^{-1}_{U,\gamma}(u)du.
$$
Note that $\psi_U=\psi_{U,\scriptsize{\begin{pmatrix}0&1\\0&0\end{pmatrix}}}$.\\
Denote $E^U(f_s)=E^{\psi_{U,0}}(f_s)$. This the constant term of $E(f_s)$ along $U(F)\backslash U({\bf A})$. If $\gamma$ is invertible, then $E^{\psi_{U,\gamma}}(f_s)=0$. This follows from the smallness of the representation generated by our Eisenstein series. In the language of unipotent orbits, when $\gamma$ is invertible, the character $\psi_{U,\gamma}$ corresponds to the partition $(2^2)$ of $4$. A similar proof to that of Prop. 5.2 in \cite{G} shows that if a Fourier coefficient corresponds to a unipotent orbit attached to a given partition $\underline{p}$ of $4$, then we must have $\underline{p}\leq (2,1^2)$. Now, our assertion follows since $(2^2)>(2,1^2)$. Thus, only $\gamma=0$ and $\gamma$ of rank one contribute to the Fourier expansion, and we have the identity
\begin{equation}\label{gen5}
E^U(f_s)=E^U(f_s)(m)+
\sum_{(\gamma_1,\gamma_2)\in B(F)\times B(F)\backslash  GL_2(F)\times GL_2(F)}\ 
\ \sum_{\alpha\in F^*}
E^{\psi_U}(f_s)\left ( j(\alpha)\begin{pmatrix} \gamma_1&\\ &\gamma_2\end{pmatrix}
m\right ).
\end{equation}
Here $j(\alpha)=\text{diag}(1,\alpha,1,1)$, and $B$ denotes the standard Borel subgroup of $GL_2$.  

Plug identity \eqref{gen5} into equation \eqref{gen3}. The constant term $E^U$ contributes zero to the expansion. Indeed, changing variables in $U$, we obtain as an inner integral $\int\psi^{-1}(x)dx$, along $F\backslash {\bf A}$, and this is zero.  As for the contribution of the second term in equation \eqref{gen5}, we first observe that the space of double cosets 
$B(F)\times B(F)\backslash  GL_2(F)\times GL_2(F)/GL_2^\Delta(F)$ 
contains two elements. Here $GL_2^\Delta$ is the diagonal embedding of $GL_2$. 
As representatives we can choose $e=(I_2,I_2)$ and $w_0=\text{diag}(J_2,I_2)$. The matrix $J_2$ was defined in the beginning of Section \ref{notation}. The contribution of $e$ is zero, again using the fact that we obtain $\int\psi(x)dx$ as an inner integration. As for $w_0$, we notice that the stabilizer of $w_0$ in $GL_2^\Delta(F)$ is $T^\Delta(F)$, where $T$ is the diagonal subgroup of $GL_2$. Hence, integral \eqref{gen3} is equal to
\begin{equation}\label{gen7}
\int\limits_{Z({\bf A})T(F)\backslash GL_2({\bf A})}
\sum_{\alpha\in F^*}\ 
\int\limits_{F\backslash {\bf A}}
\varphi_\pi(h)E^{\psi_U}(f_s)\left (j(\alpha)w_0 \begin{pmatrix} I_2& xI_2\\ &I_2\end{pmatrix}\begin{pmatrix} h&\\ &h\end{pmatrix}\right )\psi^{-1}(x)dxdh.
\end{equation}
Note that
$$
j(\alpha)w_0 \begin{pmatrix} I_2& xI_2\\ &I_2\end{pmatrix}=u\left(\begin{pmatrix}0&x\\ \alpha x&0\end{pmatrix}\right) j(\alpha)w_0;\quad \psi_U(u\left(\begin{pmatrix}0&x\\ \alpha x&0\end{pmatrix}\right))=\psi(\alpha x).
$$
Therefore conjugating in the integrand of \eqref{gen7} the unipotent matrix to the left, we obtain the integral $\int\psi^{-1}((1-\alpha)x)dx$ as inner integration. Thus, for all $\alpha\ne 1$ we get zero contribution. Factoring the measure, integral \eqref{gen7} is equal to
\begin{equation}\label{gen71}\notag
\int\limits_{T({\bf A})\backslash GL_2({\bf A})}
\int\limits_{F^*\backslash {\bf A}^*}
\varphi_\pi\left (\begin{pmatrix} a& \\ &1\end{pmatrix}h\right )
E^{\psi_U}(f_s)\left (\hat{a}w_0 \begin{pmatrix} h&\\ &h\end{pmatrix}\right )d^*adh.
\end{equation}
Here $\hat{a}=\text{diag}(1,a,a,1)$. We claim that $E^{\psi_U}(f_s)(\hat{a}m)=
|a|^{2s-1}E^{\psi_U}(f_s)(m)$. To prove this identity, let $V$ denote the maximal unipotent subgroup of $GL_4$, which contains $U$. Let $\psi_V$ denote the character of $V({\bf A})$ obtained by the trivial extension 
of $\psi_U$ from $U({\bf A})$ to $V({\bf A})$. Then 
$$
E^{\psi_U}(f_s)(m)=E^{\psi_V}(f_s)(m)=\int\limits_{V(F)\backslash V({\bf A})}
E(f_s)(vm)\psi^{-1}_V(v)dv. 
$$
This follows from the smallness of our Eisenstein series. To show this, put
$$
v(x,y)=\text{diag}(\begin{pmatrix}1&x\\&1\end{pmatrix},\begin{pmatrix}1&y\\&1\end{pmatrix}).
$$
Consider the following function on $F\backslash {\bf A} \times F\backslash {\bf A}$,
$(x,y)\mapsto E^{\psi_U}(f_s)(v(x,y)m)$ and its Fourier expansion. Its Fourier coefficients have the form
$$
\int\limits_{(F\backslash {\bf A})^2}E^{\psi_U}(f_s)(v(x,y)m)\psi^{-1}(\alpha x+\beta y)dxdy,
$$
where $\alpha, \beta\in F$. As before, if $(\alpha,\beta)\neq (0,0)$, this Fourier coefficient is zero. Only $\alpha=\beta=0$ contribute to the Fourier expansion. Unfolding the Eisenstein series, it is easy to prove that $E^{\psi_V}(f_s)(\hat{a}m)=|a|^{2s-1}E^{\psi_V}(f_s)(m)$.
From this the claim follows.

We conclude that integral \eqref{gen3} is equal to
\begin{equation}\label{gen8}
\int\limits_{T({\bf A})\backslash GL_2({\bf A})}
\left (\int\limits_{F^*\backslash {\bf A}^*}
\varphi_\pi\left (\begin{pmatrix} a& \\ &1\end{pmatrix}h\right )|a|^{2s-1}d^*a\right )
E^{\psi_V}(f_s)\left (w_0 \begin{pmatrix} h&\\ &h\end{pmatrix}\right )dh.
\end{equation}
From this the theorem follows.

\end{proof}
We remark that in \eqref{gen8} the Fourier coefficient $E^{\psi_V}(f_s)$ is Eulerian. Unfolding the Eisenstein series, we get, for $Re(s)$ sufficiently large,
\begin{equation}\label{gen9}
E^{\psi_V}(f_s)\left (w_0 \begin{pmatrix} h&\\ &h\end{pmatrix}\right )=\int\limits_{{\bf A}^2}f_s\left(\begin{pmatrix}1\\x&1\\y&0&1\\0&0&0&1\end{pmatrix}w^0\begin{pmatrix}h\\&h\end{pmatrix}\right)\psi^{-1}(y)dxdy,
\end{equation}
where $w^0=\text{diag}(J_3,1)$. Assume that $f_s$ is decomposable. Let $S$ be a finite set of places, containing the Archimedean places, outside which $f_s$ is spherical and normalized. Assume that $\varphi_\pi$ is right $GL_2(\mathcal{O}_v)$, for all $v\notin S$. Thus, for such places, it is enough to take in the local integration over $T(F_v)\backslash GL_2(F_v)$ in \eqref{gen8}, $h=\begin{pmatrix}1&z\\&1\end{pmatrix}$. Denote the local factor of $f_s$, at $v\notin S$, by $f^0_{s,v}$. Assume also that, for such $v$, $\psi_v$ is normalized. Then it is not hard to see that the following local integral of \eqref{gen9}, at $v$,
\begin{equation}\label{gen10}
\int\limits_{ F_v^2}f^0_{s,v}\left(\begin{pmatrix}1\\x&1\\y&0&1\\0&0&0&1\end{pmatrix}w^0\begin{pmatrix}1&z\\&1\\&&1&z\\&&&1\end{pmatrix}\right)\psi_v^{-1}(y)dxdy,
\end{equation}
is supported in $z\in \mathcal{O}_v$, and the integration is supported in $x\in \mathcal{O}_v$, so that for such $z$, \eqref{gen10} is equal to
\begin{equation}\label{gen11}
\int\limits_{ F_v}f^0_{s,v}\left(\begin{pmatrix}1\\0&1\\y&0&1\\0&0&0&1\end{pmatrix}\right)\psi_v^{-1}(y)dy.
\end{equation}
It is easy to compute \eqref{gen11}. It is equal to $1-q_v^{-4s}$, where $q_v$ is the number of elements in the residue field of $F_v$. Thus, the product of local integrals \eqref{gen11}, over all $v\notin S$ (and $\text{Re}(s)>1$) is $(L^S(4s))^{-1}$.

\section{Examples of new way type integrals}\label{new}
In this section we consider two examples. The first example is the famous construction of global integrals given by Piatetski-Shapiro and Rallis in \cite{PS-R2}. This was the first example of the the so-called New Way type integrals (after the title of the paper \cite{PS-R2}). These integrals represent the standard  $L$-function $L(\pi,s)$ of an irreducible, automorphic, cuspidal representation $\pi$ of $Sp_{2k}({\bf A})$.
In the second example  we extend the above construction and introduce a (new) New Way type of integrals which represent the standard $L$-function $L(\pi\times\tau,s)$. Here $\tau$ is an irreducible, automorphic, cuspidal representation of $GL_n({\bf A})$. We will do it by considering the case $k=n=2$. This will demonstrate how can one use the
procedure described in the introduction to actually construct new global integrals which we conjecture to be Eulerian.

\subsection{The New Way Construction of Piatetski-Shapiro and Rallis}\label{psr1}
First, we describe the functional used in the integrals in \cite{PS-R2}.
Let $\pi$ denote an irreducible, automorphic, cuspidal representation of $Sp_{2k}({\bf A})$. Let
$T_0$ denote a symmetric matrix in $GL_k(F)$. It is not essential, but we may assume that $T_0$ is diagonal.  Denote $T=J_kT_0$. With these notations, the functional used in the New Way integral is given by
\begin{equation}\label{new1}
\varepsilon_T(\varphi_\pi)=\int\limits_{Mat_k^0(F)\backslash Mat_k^0({\bf A})}\varphi_\pi
\left (\begin{pmatrix} I_k&Z\\ &I_k\end{pmatrix} \right )\psi^{-1}(\text{tr}(TZ))dZ
\end{equation}
To describe the global integral, we assume that $k$ is even. This assumption is also made in \cite{PS-R2} at the beginning of Section 2. It is made to avoid the use of Eisenstein series on metaplectic groups. Denote by $\chi_T$ the quadratic character of ${\bf A}^*$,  $\chi_T(x)=(x,\text{det}(T))$. Here, $(,)$ is the global Hilbert symbol. This is the quadratic character corresponding to $\text{disc}(T_0)$.

Let $SO_{T_0}$ denote the special orthogonal group in $k$ variables over $F$, corresponding to $T_0$. Consider the dual pair $SO_{T_0}\times Sp_{2k}$ inside $Sp_{2k^2}$. Since $k$ is even, $SO_{T_0}({\bf A})\times Sp_{2k}(\bf A)$ splits in $Sp^{(2)}_{2k^2}(\bf A)$. Denote by $\omega_{\psi_T}$ the restriction of the Weil representation $\omega_\psi$ of $Sp^{(2)}_{2k^2}(\bf A)$, corresponding to $\psi$, to the image of $Sp_{2k}(\bf A)$ under the splitting. It can be realized in $\mathcal{S}(Mat_k({\bf A}))$, such that for $\phi\in \mathcal{S}(Mat_k({\bf A}))$, $g\in Sp_{2k}(\bf A)$,
$$
\omega_{\psi_T}\left (\begin{pmatrix} I_k&Z\\ &I_k\end{pmatrix} g\right )\phi(I_k)=\psi(\text{tr}(TZ))\omega_{\psi_T}(g)(I_k).
$$
See \cite{PS-R2}, p. 117. Denote by  $\theta_{\psi_T}^\phi$ the corresponding theta series, restricted to $Sp_{2k}(\bf A)$.
More generally, we have, for $\phi\in \mathcal{S}(Mat_k({\bf A}))$, the theta series $\theta^\phi_{\psi,2k^2}$ on $Sp^{(2)}_{2k^2}(\bf A)$ corresponding to $\phi$. Denote by $i_T$ the above splitting of $SO_{T_0}({\bf A})\times Sp_{2k}(\bf A)$. Then, for $(m,g)\in SO_{T_0}({\bf A})\times Sp_{2k}(\bf A)$,  $\phi\in \mathcal{S}(Mat_k({\bf A}))$,
$$
\theta^\phi_{\psi,2k^2}(i_T(m,g))=\sum_{x\in Mat_k(F)}\omega_{\psi_T}(g)\phi(m^{-1}x).
$$ 
We should denote $i_{T,2k}$, but we will drop $2k$, as it will be clear from the context. We have, for $g\in Sp_{2k}({\bf A})$, $\theta_{\psi_T}^\phi(g)=\theta^\phi_{\psi,2k^2}(i_T(1,g))$. 

Let $E(\eta_{\chi_T,s})(g)$ denote an Eisenstein series on $Sp_{2k}({\bf A})$, attached to a smooth, holomorphic section $\eta_{\chi_T,s}$ of $Ind_{Q_k({\bf A})}^{Sp_{2k}({\bf A})}(\chi_T\circ \text{det})|\text{det}\cdot|^s$. 
Then the global integral introduced in \cite{PS-R2}, equation (2.1), is given by
\begin{equation}\label{new2}
\int\limits_{Sp_{2k}(F)\backslash Sp_{2k}({\bf A})}\varphi_\pi(g)
\theta_{\psi_T}^\phi(g)E(\eta_{\chi_T,s})(g)dg.
\end{equation}
Let $S$ be a finite set of places of $F$, containing the Archimedean places, such that $\pi_v$ is unramified for $v\notin S$, and the diagonal coordinates of $T_0$ are units outside $S$. Then the integral \eqref{new2} represents outside $S$ the ratio 
\begin{equation}\label{new2.1}
\frac{L^S(\pi,s+\frac{1}{2})}{d^S_k(\chi_T,s)},
\end{equation}
where 
\begin{equation}\label{new2.2}
d^S_k(\chi_T,s)=L^S(\chi_T,s+\frac{k+1}{2})\prod_{i=0}^{\frac{k}{2}-1}L^S(2s+2i+1).
\end{equation}

In \cite{PS-R1}, the authors construct a doubling integral, which represents outside $S$ the ratio 
\begin{equation}\label{new2.3}
\frac{L^S(\pi,s+\frac{1}{2})}{d^S_{2k}(1,s)},
\end{equation}
where $d^S_{2k}(1,s)$ is obtained from \eqref{new2.2}, with $2k$ instead of $k$ and the trivial character instead of $\chi_T$. This construction uses an Eisenstein series ${\mathcal E}(f_s)$ on $Sp_{4k}({\bf A})$, corresponding to a smooth, holomorphic section $f_s$ of 
$Ind_{Q_{2k}({\bf A})}^{Sp_{4k}({\bf A})}|\text{det}\cdot|^s$. 
As explained in Section \ref{not3}, for the case $n=1$, the function
\begin{equation}\label{new3}
\xi(\varphi_\pi,f_s)(g)=\int\limits_{Sp_{2k}(F)\backslash Sp_{2k}({\bf A})}\varphi_\pi(h){\mathcal E}(f_s)\left (\begin{pmatrix} g_1&&g_2\\ &h&\\ g_3&&g_4\end{pmatrix}\right )dh
\end{equation}
is a cusp form in the space of $\pi^\iota$. Here $g=\begin{pmatrix} g_1&g_2\\ g_3&g_4\end{pmatrix}\in Sp_{2k}(\bf A)$. 

Our goal in this example is to  compute integral \eqref{new1} with $(\varphi_\pi)^\iota$ instead of $\varphi_\pi$, and then replace $(\varphi_\pi)^\iota$  by $\xi(\varphi_\pi,f_s)$.  In other words, we compute the integral 
\begin{equation}\label{new4}
\int\limits_{Sp_{2k}(F)\backslash Sp_{2k}({\bf A})}
\int\limits_{Mat_k^0(F)\backslash Mat_k^0({\bf A})}
\varphi_\pi(h){\mathcal E}(f_s)\left (\begin{pmatrix} I_k&&Z\\ &h& \\ &&I_k\end{pmatrix}\right )\psi^{-1}(\text{tr}(TZ))dZdh.
\end{equation}
We prove
\begin{theorem}\label{prop2}
Given $\phi\in \mathcal{S}(Mat_k({\bf A}))$, there are nontrivial choices of the sections $f_s$, $\eta_{\chi_T,s}$, such that integral \eqref{new4} is equal to integral \eqref{new2}, for any cusp form $\varphi_\pi$ in the space of $\pi$. The choice of sections is made explicit in the proof, and is such that $f_s$ is a certain convolution of any given section $f'_s$ by a Schwartz function depending on $\phi$ and another given Schwartz function $\phi_2\in \mathcal{S}(Mat_k({\bf A}))$; $\eta_{\chi_T,s}$ is given by an explicit (integral) formula $\eta_{\chi_T,s}=\eta(f'_s,\phi_2)$.
\end{theorem}

\begin{proof}

Recall that we assume that $k$ is even. Consider the the unipotent radical $U_{k,4k}$, which we denote for short by $U$. It consists of all matrices of the form
\begin{equation}\label{new5}
u=\begin{pmatrix} I_k&a&b&Z\\ &I_k&&\star\\ &&I_k&\star\\ &&&I_k\end{pmatrix}\in Sp_{4k}.
\end{equation}
This group is a generalized Heisenberg group, as follows. Consider the Heisenberg group ${\mathcal H}_{2k^2+1}$ realized as $Mat_k\times Mat_k\times Mat_1$. Then there is a  homomorphism from $U$ onto ${\mathcal H}_{2k^2+1}$. We choose the following homomorphism,
$l_T(u)=(a,b,\text{tr}(TZ))$. 
Embed $SO_{T_0}\times Sp_{2k}$ inside $Sp_{4k}$  by $(m,h)\to \text{diag}(m,h,m^*)\in Sp_{4k}$.
Note that the action by conjugation of $\text{diag}(m,h,m^*)$ on the element \eqref{new5} takes $(a,b,Z)$ to $((m^{-1}a,m^{-1}b)h,m^{-1}ZJ_k{}^tm^{-1}J_k)$. Denote, for short, $(m,h)=\text{diag}(m,h,m^*)$. Let $D$ denote the semi-direct product of $U$ and the image of $SO_{T_0}\times Sp_{2k}$ inside $Sp_{4k}$.  Consider, for $d\in D(\bf A)$, the Fourier coefficient 
\begin{equation}\label{new6}
{\mathcal E}^{\psi_T}(f_s)(d)=\int\limits_{Mat_k^0(F)\backslash Mat_k^0({\bf A})}
{\mathcal E}(f_s)\left (\begin{pmatrix} I_k&&Z\\ &I_{2k}& \\ &&I_k\end{pmatrix}d\right )\psi^{-1}(\text{tr}(TZ))dZ.
\end{equation}
It defines a function of  $d\in D(F)\backslash D({\bf A})$. Now we use a theorem of Ikeda on Fourier-Jacobi coefficients. See \cite{I}, Sec. 1. It says that the following set of functions of $d=v(m,h)$, $v\in U({\bf A})$, $m\in SO_{T_0}(\bf A)$, $h\in Sp_{2k}(\bf A)$, 
\begin{equation}\label{new7.0}
\theta_{\psi,2k^2}^{\phi_1}(l_T(v)i_T(m,h))
\int\limits_{U(F)\backslash U({\bf A})}
\overline{\theta_{\psi,2k^2}^{\phi_2}(l_T(u)i_T(m,h))}{\mathcal E}(f_s)(u(m,h))du
\end{equation}
spans a dense subspace of the space of functions given by integrals \eqref{new6}. In more details, Ikeda introduced a family of functions $\varphi_{\phi_1,\phi_2}$ on ${\mathcal H}_{2k^2+1}(\bf A)$, which transform by $\psi^{-1}$ under translations by the center, and such that $(x,y)\mapsto \varphi_{\phi_1,\phi_2}(x,y,0)$ lies in $\mathcal{S}(Mat_k({\bf A})\times Mat_k({\bf A}))$. Here $\phi_1,\phi_2\in \mathcal{S}(Mat_k({\bf A})$. See \cite{I}, p. 621. Let, for $v\in U(\bf A)$, $m\in SO_{T_0}(\bf A)$, $h\in Sp_{2k}(\bf A)$,
\begin{equation}\label{new7.1}
\rho(\varphi_{\phi_1,\phi_2}){\mathcal E}^{\psi_T}(f_s)(v(m,h))=
\int\limits_{C({\bf A})\backslash U({\bf A})}\varphi_{\phi_1,\phi_2}(l_T(u)){\mathcal E}^{\psi_T}(f_s)(v(m,h)u)du,
\end{equation}
where $C$ is the center of $U$. Then the functions \eqref{new7.1} span a dense subspace of the closure in the Frechet topology of the functions $v(m,h)\mapsto {\mathcal E}^{\psi_T}(f_s)(v(m,h))$. Here $s$ is fixed away from the set of poles of our Eisenstein series, and we let the section vary. The proof of Prop. 1.3 in \cite{I} shows that the r.h.s. of \eqref{new7.1} is equal to \eqref{new7.0}. Note that \eqref{new4} can be rewritten as
\begin{equation}\label{new4.1}
\int\limits_{Sp_{2k}(F)\backslash Sp_{2k}({\bf A})}
\varphi_\pi(h){\mathcal E}^{\psi_T}(f_s) ((1,h))dh.
\end{equation}
Let us realize $C({\bf A})\backslash U({\bf A})$ as the subset which is the product of subgroups $U_1(\bf A)$, $U_2(\bf A)$ where $U_1$ (resp. $U_2$) is the subgroup of elements $u$ of the form \eqref{new5}, such that $Z=0$ and $b=0$ (resp. $a=0$). Then 
\begin{equation}\label{new7.2}
f^{\phi_1,\phi_2}_s=\int\limits_{U_2({\bf A})}\int\limits_{U_1({\bf A})}\varphi_{\phi_1,\phi_2}(l_T(u_2)l_T(u_1))\rho(u_2)\rho(u_1)f_sdu_1du_2
\end{equation}
is a smooth, holomorphic section of $Ind_{Q_{2k}({\bf A})}^{Sp_{4k}({\bf A})}|\text{det}\cdot|^s$; $\rho$ denotes a right translation. Now \eqref{new7.1} reads as 
\begin{equation}\label{new7.3}
\rho(\varphi_{\phi_1,\phi_2}){\mathcal E}^{\psi_T}(f_s)(v(m,h))=
{\mathcal E}^{\psi_T}(f^{\phi_1,\phi_2}_s)(v(m,h)).
\end{equation}

Thus, given the section $f_s$ and $\phi_1,\phi_2\in \mathcal{S}(Mat_k({\bf A}))$, we construct the section \eqref{new7.2}, and substitute in the integral \eqref{new4} (or in \eqref{new4.1}) $f^{\phi_1,\phi_2}_s$ instead of $f_s$. By \eqref{new7.3} and what we explained before, we get
\begin{equation}\label{new8}
\int\limits_{Sp_{2k}(F)\backslash Sp_{2k}({\bf A})}
\varphi_\pi(h)\theta_{\psi,2k^2}^{\phi_1}(i_T(1,h))
\int\limits_{U(F)\backslash U({\bf A})}
\overline{\theta_{\psi,2k^2}^{\phi_2}(l_T(u)i_T(1,h))}{\mathcal E}(f_s)(u(1,h))dudh.
\end{equation}
Note that $\theta_{\psi,2k^2}^{\phi_1}(i_T(1,h))=\theta_T^{\phi_1}(h)$. From \cite{I}, Theorem 3.2, the inner $du$-integral in \eqref{new8} is an Eisenstein series $E(\eta(f_s,\phi_2))(h)$ on $Sp_{2k}(\bf A)$, corresponding to an explicitly written section $\eta(f_s,\phi_2)$ of $Ind_{Q_k({\bf A})}^{Sp_{2k}({\bf A})}(\chi_T\circ \text{det})|\text{det}\cdot|^s$. Thus \eqref{new8} is equal to 
\begin{equation}\label{new9}
\int\limits_{Sp_{2k}(F)\backslash Sp_{2k}({\bf A})}
\varphi_\pi(h)\theta_T^{\phi_1}(h)E(\eta(f_s,\phi_2))(h)dh,
\end{equation}
which is an integral of the form \eqref{new2}. This completes the proof of the theorem.

\end{proof}

Let us examine in detail the equality of the last theorem, using the choice of data made in the proof. Assume that $\varphi_\pi$ corresponds to a decomposable vector, which is unramified outside $S$. Similarly, assume that $f_s=\prod_v f_{v,s}$ is a product of local sections, which are unramified and normalized outside $S$. Assume also that the Schwartz functions $\phi_1$, $\phi_2$ in the proof are decomposable $\phi_i=\prod_v\phi_{i,v}$, $i=1,2$, where, for $v\notin S$, $\phi_{i,v}=\phi^0$ -the characteristic function of $Mat_k(\mathcal{O}_v)$. It is easy to check that $f^{\phi_1,\phi_2}_s$ is decomposable as the product of the analogous local sections $f^{\phi_{1,v},\phi_{2,v}}_{v,s}$ , and that for $v\notin S$ the corresponding factor is the normalized unramified section. Using the notation of \eqref{dou6.6}, the integral \eqref{new4}, with $f_s$ replaced by $f^{\phi_1,\phi_2}_s$, is equal to
\begin{equation}\label{new9.1}
\frac{L^S(\pi\times\tau,s+\frac{1}{2})}{d_{2k}^S(1,s)}
\varepsilon_T((\iota\circ \ell)(\otimes_{v\in S}\xi(\varphi_{\pi_v},f^{\phi_{1,v},\phi_{2,v}}_{v,s})\otimes \varphi_\pi^S)).
\end{equation}
This is a meromorphic function. Denote $\xi_S(\varphi_\pi,f^{\phi_1,\phi_2}_s)=\otimes_{v\in S}\xi(\varphi_{\pi_v},f^{\phi_{1,v},\phi_{2,v}}_{v,s})$. We proved that \eqref{new9} is equal to \eqref{new9.1}, that is 
\begin{multline}\label{new9.2}
\int\limits_{Sp_{2k}(F)\backslash Sp_{2k}({\bf A})}
\varphi_\pi(h)\theta_T^{\phi_1}(h)E(\eta(f_s,\phi_2))(h)dh=\\
\frac{L^S(\pi\times\tau,s+\frac{1}{2})}{d_{2k}^S(1,s)}
\varepsilon_T((\iota\circ \ell)(\xi_S(\varphi_\pi,f^{\phi_1,\phi_2}_s)\otimes \varphi_\pi^S)).
\end{multline}
Examining the section $\eta(f_s,\phi_2)$, we see that it is decomposable and has the form 
$$
\frac{d^S_k(\chi_T,s)}{d_{2k}^S(1,s)}(\otimes_{v\in S}\eta_v(f_{v,s},\phi_{2,v})\otimes \eta^S_{\chi_T,s}),
$$
where $\eta_v(f_{v,s},\phi_{2,v})$ is the local section of $Ind_{Q_k(F_v)}^{Sp_{2k}(F_v)}(\chi_{T,v}\circ \text{det})|\text{det}\cdot|^s$, which is the local analog at $v$ of $\eta(f_s,\phi_2)$ defined in \cite{I}, Theorem 3.2; $\eta^S_{\chi_T,s}=\otimes_{v\notin S}\eta^0_{\chi_{T,v},s}$, where, for $v\notin S$, $\eta^0_{\chi_{T,v},s}$ is the normalized, spherical section of $Ind_{Q_k(F_v)}^{Sp_{2k}(F_v)}(\chi_{T,v}\circ \text{det})|\text{det}\cdot|^s$. Denote $\eta_S(f_s,\phi_2)=\otimes_{v\in S}\eta_v(f_{v,s},\phi_{2,v})$. Then \eqref{new9.2} can be rewritten as
\begin{multline}\label{new9.3}
\int\limits_{Sp_{2k}(F)\backslash Sp_{2k}({\bf A})}
\varphi_\pi(h)\theta_T^{\phi_1}(h)E(\eta_S(f_s,\phi_2)\otimes \eta^S_{\chi_T,s})(h)dh=\\
\frac{L^S(\pi\times\tau,s+\frac{1}{2})}{d_{2k}^S(1,s)}
\varepsilon_T((\iota\circ \ell)(\xi_S(\varphi_\pi,f^{\phi_1,\phi_2}_s)\otimes \varphi_\pi^S)).
\end{multline} 
 This explains the work in \cite{PS-R2}, namely the theorem that the global integral \eqref{new2} represents \eqref{new2.1}.

\subsection{Extending the Construction of Piatetski-Shapiro and Rallis}\label{psr2}
We keep the notations of the previous sub-section. In this section we take $k=2$. Thus, $\pi$ is an irreducible, automorphic, cuspidal representation of $Sp_4({\bf A})$. Let $\tau$ denote an irreducible, automorphic, cuspidal representation of $GL_2({\bf A})$. Our goal is to introduce a new global integral which will unfold to an integral involving the Fourier coefficient given by integral \eqref{new1}. We  then conjecture that this integral represents the standard $L$-function $L(\pi\times\tau,s)$. This is an example showing how to use our procedure to obtain new global integrals representing an $L$-function coming from doubling integrals.

First we need some preliminaries. Consider the unipotent radical $U_{2^2,16}$. It consists of all matrices of the form
\begin{equation}\label{new50}
\begin{pmatrix} I_2&a_1&a_2&a_3&a_4\\ &I_2&y&Z&a_3'\\ &&I_8&y'&a_2'
\\&&&I_2&a_1'\\ &&&&I_2\end{pmatrix}\in Sp_{16}.
\end{equation}
Let $U^0_{2^2,16}$ denote the subgroup of $U_{2^2,16}$ consisting of all matrices of the form \eqref{new50}, such that $y=\begin{pmatrix} 0_{2\times 6},&y_0
\end{pmatrix}$, $y_0\in Mat_2$. Notice that $U_{2^2,16}$ has a structure of a generalized Heisenberg group. Recall the matrix $T=J_2T_0$ from the previous sub-section. Write $T_0=diag(t_1,t_2)$. Define $l_T^0 :  U_{2^2,16}\to {\mathcal H}_{17}$ by 
$l_T^0(u)=(y,\text{tr}(TZ))\in {\mathcal H}_{17}$, where $u$ is in the form \eqref{new50}.
Consider the dual pair $SO_{T_0}\times Sp_8$ inside $Sp_{16}$. Its Adele points split in $Sp^{(2)}_{16}({\bf A})$. We may realize the Weil representation $\omega_{\psi,16}$ of $Sp^{(2)}_{16}({\bf A})$, corresponding to $\psi$, in $\mathcal{S}(Mat_{2\times 4}({\bf A}))$. Consider, for $\phi\in \mathcal{S}(Mat_{2\times 4}({\bf A}))$, the corresponding theta series $\theta_{\psi,16}^\phi$. We may take a splitting $i^0_T$ of $SO_{T_0}({\bf A})\times Sp_8({\bf A})$, such that $\omega_{\psi,16}(i^0_T(m,1))\phi(x)=\phi(m^{-1}x)$, for $m\in SO_{T_0}({\bf A})$, $x\in Mat_{2\times 4}({\bf A})$.
 
Let $\psi_{U_{2^2,16}}$ denote the character of $U_{2^2,16}$ defined by $\psi_{U_{2^2,16}}(u)=\psi(
\text{tr}(a_1))$. Let  $\psi_{U^0_{2^2,16},T}$ denote the character of $U^0_{2^2,16}({\bf A})$ defined by $\psi_{U^0_{2^2,16},T}(u)=\psi(\text{tr}(a_1))\psi(\text{tr}(TZ))\psi^{-1}(\text{tr}(y_0))$.
Again, we wrote $u$ in the form \eqref{new50}.

Let $f_{\Delta(\tau,4),s}$ be a smooth, holomorphic section of 
$Ind_{Q_8({\bf A})}^{Sp_{16}({\bf A})}\Delta(\tau,4)|\text{det}\cdot|^s$, and let $E(f_{\Delta(\tau,4),s})$ be the corresponding Eisenstein series on $Sp_{16}({\bf A})$. Consider the following Fourier-Jacobi coefficient of $E(f_{\Delta(\tau,4),s})$,
\begin{equation}\label{new51}
\int\limits_{U_{2^2,16}(F)\backslash U_{2^2,16}({\bf A})}
\theta_{\psi,16}^\phi(l_T^0(u)i^0_T(h_0,g_0))E(f_{\Delta(\tau,4),s})(u{}^d(h_0,g_0))
\psi^{-1}_{U_{2^2,16}}(u)du.
\end{equation}
Here, ${}^d(h_0,g_0)=\text{diag}(h_0,h_0,g_0,h_0^*,h_0^*)$, $\phi\in \mathcal{S}(Mat_{2\times 4}({\bf A}))$, $h_0\in SO_{T_0}({\bf A})$ and $g_0\in Sp_8({\bf A})$. Recall the quadratic character $\chi_T$ of $F^*\backslash {\bf A}^*$ ; $\chi_T(x)=(x, \text{det}(T))=(x,-t_1t_2)$. We have
\begin{lemma}\label{lem15}
Fix $h_0\in SO_{T_0}({\bf A})$. Then, as a function of $g_0\in Sp_8({\bf A})$, the integral \eqref{new51} is an Eisenstein series $E_{\tau\otimes\chi_T,2}(g_0,s)$ as in the beginning of Sec. \ref{not2}. More precisely, there is a smooth, meromorphic section $\lambda(f_{\Delta(\tau,4),s},\phi)$ of $Ind_{Q_4({\bf A})}^{Sp_8({\bf A})}\Delta(\tau\otimes\chi_T,2)|\text{det}\cdot|^s$, such that the integral \eqref{new51} is equal to the Eisenstein series $E(\lambda(f_{\Delta(\tau,4),s},\phi))$.
\end{lemma}
\begin{proof}
This lemma can be proved in two ways. The first proof uses the identities 
\eqref{id1} and \eqref{id2}. Indeed, after some root exchange process, we can use Identity \eqref{id1} to obtain the Eisenstein series $E_{\tau,3}^{(2)}(\cdot,s)$ defined on $Sp_{12}^{(2)}({\bf A})$. Then, using more root exchange, and Identity \ref{id2}, the proof will follow. 
	
A second approach is to prove the lemma directly by unfolding the Eisenstein series. We will give some details for this approach, and to simplify the computations, we will assume that $SO_{T_0}({\bf A})$ is anisotropic. In other words, we assume that $-t_1t_2$ is not a square in $F$. 
	
Set $h_0=1$. We start by unfolding the Eisenstein series in \eqref{new51}, assuming that $Re(s)$ is sufficienly large. We consider the space of double cosets $Q_8(F)\backslash Sp_{16}(F)/U_{2^2,16}(F){}^d(1,Sp_8(F))$. See Sec. \ref{not2} for notations. It follows from the Bruhat decomposition that all representatives can be chosen in the form
\begin{equation}\label{new52}
\gamma=wx_{\alpha_1}(c_1)x_{\alpha_3}(c_3)\ \ \ \ \ \ \ \ 
w=\begin{pmatrix} \epsilon_1&&\epsilon_2\\ &I_8&\\ \epsilon_3&&\epsilon_4\end{pmatrix}.
\end{equation}
Here, $x_{\alpha_1}(c_1)=I_{16}+c_1e'_{1,2}$, and $x_{\alpha_3}(c_3)=I_{16}+c_3e'_{3,4}$ where $c_i\in F$.  Also, $w$ is a Weyl element in $Sp_{16}(F)$, which is assumed to be a permutation matrix with nonzero entries $\pm 1$. We can choose the matrix  $\epsilon_1$ to be a diagonal matrix, and $\epsilon_2$ can be chosen so that all nonzero entries are on the other diagonal. Moreover, we can assume that all nonzero entries of $\epsilon_1$ and $\epsilon_2$ are ones. These conditions determine $w$ uniquely. 
	
Let $U_8$  denote the unipotent radical of $Q_8$. Let $\gamma$ denote a representative as in \eqref{new52}. If for some $u\in U_{2^2,16}({\bf A})$, such that $l^0_T(u)$ lies in the center of $\mathcal{H}_{17}({\bf A})$, we have $\psi_{U^0_{2^2,16},T}(u)\ne 1$, and $\gamma u\gamma^{-1}\in U_8({\bf A})$, then the double coset of $\gamma$ contributes zero to the integral \eqref{new51}. Note that $u$ as above lies in $U^0_{2^2,16}({\bf A})$. For $1\le i\le 4$, let $\epsilon_1(i)$ denote the $(i,i)$ entry of $\epsilon_1$. Assume that $\epsilon_1(3)=1$. Thus, for 
$u=I_{16}+de_{3,14}$ we have $\psi_{U^0_{2^2,16},T}(u)\ne 1$, and 
$\gamma u\gamma^{-1}\in U_8({\bf A})$. Hence, we may assume that 
$\epsilon_1(3)=0$, and that the $(2,3)$ entry of $\epsilon_2$ is one. Similarly, if $\epsilon_1(4)=1$, we can then chose a suitable matrix 
$u=I_{16}+d_1e_{4,13}+d_2e'_{3,13}+d_3e_{3,14}$ such that $\psi_{U^0_{2^2,16},T}(u)\ne 1$, and  $\gamma u\gamma^{-1}\in U_8({\bf A})$. We mention that at this point we use the fact that $-t_1t_2$ is not a square. 
	
We conclude that $\epsilon_1(4)=0$. If $\epsilon_1(1)=1$, we use 
$u=I_{16}+de'_{1,3}$, and if $\epsilon_1(2)=1$ we use $u=I_{16}+d_1e'_{2,4}+d_2e'_{1,4}$. We omit the details. Thus, we deduce that $\gamma$ is such that $\epsilon_1=0$. Matrix multiplication implies that we may assume that $\gamma$ in this case can be chosen so that $c_1=c_3=0$. Modifying the representative $\gamma$, we deduce that the only possible double coset representative which can contribute a nonzero term to integral \eqref{new51} is the containing the following representative
$$
w_0=\begin{pmatrix} &I_4&&\\ &&&I_4\\ -I_4&&&\\ &&I_4&\end{pmatrix}.
$$
	
Computing the stabilizer for $w_0$, we deduce that, for $Re(s)$ large enough, the integral \eqref{new51} is equal to
\begin{multline}\label{new53}
\sum_{\gamma\in Q_4(F)\backslash Sp_8(F)}\ \ \int\limits_{U^1({\bf A})}
\int\theta_{\psi,16}^\phi((0_{2\times 4},b_2,0)l_T^0(u_1)i^0_T(1,\gamma g_0))\\ 
f_{\Delta(\tau,4),s}\left (\begin{pmatrix} I_6&x&&\\ &I_2&&\\ &&I_2&x'\\ &&&I_6
\end{pmatrix}w_0\hat{b}_2u_1 {}^d(1,\gamma g_0)\right )\widetilde{\psi}
(x)\psi^{-1}_{U_1,T}(u_1)db_2dxdu_1.
\end{multline}
We explain the notation. We start with the definition of the group $U^1$. It is the subgroup of $U_{2^2,16}$ consisting of all matrices of the form \eqref{new5} with $k=4$, such that $b=0$. Notice that the projection $l_T^0$ is non-trivial on $U^1$; the character $\psi_{U^1,T}$ is defined as follows. Write the matrix $Z$ in \eqref{new5} as $Z=\begin{pmatrix} *&*\\ Z_1&*\end{pmatrix}$, where $Z_1\in Mat_2({\bf A})$. Then  $\psi_{U^1,T}(u_1)=\psi(\text{tr}(TZ_1))$. The integration domain for the $x$ variable is $Mat_{6\times 2}(F)\backslash 
Mat_{6\times 2}({\bf A})$. The character $\widetilde{\psi}$ is defined as follows. Write $x=\begin{pmatrix} x_1\\ x_2\end{pmatrix}$, where
$x_1\in Mat_{4\times 2}({\bf A})$ and $x_2\in Mat_2({\bf A})$. Then  
$\widetilde{\psi}(x)=\psi^{-1}(\text{tr}(x_2))$. Finally, the variable $b_2$
is integrated over $Mat_{2\times 4}(F)\backslash 
Mat_{2\times 4}({\bf A})$. It is embedded in $Sp_{16}$ as all matrices in \eqref{new5} with $a=0$, $Z=0$, and $b=\begin{pmatrix} 0_{2\times 4}\\ b_2\end{pmatrix}$. 
	
We claim that for all $h\in Sp_{16}({\bf A})$, we have
\begin{multline}\label{new54}
\int f_{\Delta(\tau,4),s}\left (\begin{pmatrix} I_6&x&&\\ &I_2&&\\ &&I_2&x'\\ &&&I_6
\end{pmatrix}h\right )\widetilde{\psi}(x)dx =\\ 
=\int f_{\Delta(\tau,4),s}\left (\begin{pmatrix} I_4&x_3&x_1&&&\\ &I_2&x_2&&&\\ &&I_2&&&\\ &&&I_2&x_2'&\star\\ &&&&I_2&x'_3\\ &&&&&I_4
\end{pmatrix}h\right )\psi^{-1}(tr(x_2))dx_i.
\end{multline}
Here, the integral on the left hand side is integrated as in integral 
\eqref{new53}. The integral on the right hand side has an extra integration over the variable $x_3$ which is integrated over 
$Mat_{4\times 2}(F)\backslash  Mat_{4\times 2}({\bf A})$. The proof of this identity is exactly as the proof of Proposition 2.4 in \cite{G-S}. 
For simplicity, we shall denote by $f^{\psi'}_{\Delta(\tau,4),s}(h)$ the integral on the right hand side of equation \eqref{new54}. 
	
Plugging this identity in integral \eqref{new53}, we then conjugate the matrix $\hat{b}_2$ across $w_0$, and then we change variables in $x_3$. Thus, integral \eqref{new51} is equal to 
	\begin{equation}\label{new55}\notag
	\sum_{\gamma\in Q_4(F)\backslash Sp_8(F)}\ \ \int\limits_{U^1({\bf A})}
	\int\theta_{\psi,16}^\phi((0_{2\times 4},b_2,0)l_T^0(u_1)i^0_T(1,\gamma g_0))db_2\\ 
	f^{\psi'}_{\Delta(\tau,4),s}\left (w_0u_1{}^d(1,\gamma g_0)\right )\psi^{-1}_{U^1,T}(u_1)du_1.
	\end{equation} 
	Unfolding the theta series and integrating over $b_2$, we deduce that integral \eqref{new51} is equal to (for $Re(s)$ large)
	\begin{equation}\label{new56}
	\sum_{\gamma\in Q_4(F)\backslash Sp_8(F)}\ \ \int\limits_{U^1({\bf A})}
	\omega_{\psi,16}(l_T^0(u_1)i^0_T(1,\gamma g_0))\phi(0)
	f^{\psi'}_{\Delta(\tau,4),s} (w_0u_1{}^d(1,\gamma g_0))\psi^{-1}_{U^1,T}(u_1)du_1.
	\end{equation}
	
		We claim that the integral 
	\begin{equation}\label{new57}\notag
	\lambda(f_{\Delta(\tau,4),s},\phi)=\int\limits_{U^1({\bf A})}
	\omega_{\psi,16}(l_T^0(u_1)i^0_T(1,g_0))\phi(0)
	f^{\psi'}_{\Delta(\tau,4),s} (w_0u_1{}^d(1,g_0))\psi^{-1}_{U^1,T}(u_1)du_1
	\end{equation}
	is a section of the induced representation $Ind_{Q_4({\bf A})}^{Sp_8({\bf A})}\Delta(\tau\otimes\chi_T,2)|\text{det}\cdot|^s$. To prove this we consider a matrix of the form $g_0=\begin{pmatrix} A&B\\
	&A^*\end{pmatrix}\in Q_8({\bf A})$ where $A\in GL_4({\bf A})$. It is easy to check that the integral is left invariant under 
	$g=\begin{pmatrix} I_4&B\\ &I_4\end{pmatrix}$. Next plug $g_0=\begin{pmatrix} A&\\ &A^*\end{pmatrix}$, and conjugate it to the left. 
	We obtain $\chi_T(\text{det} A)|\text{det} A|^{-4}$ from the change of variables in $u_1$. Then we obtain a factor of $|\text{det} A|$ from the action of the Weil representation, and  another factor of $|\text{det} A|^s\delta_{Q_8}^{1/2}(\text{diag}(A,I_8,A^*))=|\text{det} A|^{s+9/2}$ from the section $f_{\Delta(\tau,4),s}$. Finally, notice that the integral $f^{\psi'}_{\Delta(\tau,4),s}$ contains the constant term along the standard parabolic subgroup $L$ of $GL_8$ whose Levi part is isomorphic to $GL_4\times GL_4$. This constant term contributes 
	the factor $\Delta(\tau,2)|\text{det} A|^{-1}\delta_{L}^{1/2}(\text{diag}(A,I)=
	\Delta(\tau,2)|\text{det} A|$. Combining all these factors we get that $g_0$ is transformed as $\Delta(\tau\otimes \chi_T,2)|\cdot|^s\delta_{Q_4}^{1/2}$. 
	
	\end{proof}

To introduce the new integral construction for the tensor product $L$-function $L(\pi\times\tau,s)$, we start with the global doubling construction for this $L$-function  given in \cite{C-F-G-K1}. Then we write down the corresponding function $\xi(\varphi_\pi,f_{\Delta(\tau,4),s})$ given by integral \eqref{dou3}, and finally compute the Fourier coefficient of $\xi(\varphi_\pi,f_{\Delta(\tau,4),s})$ given by integral \eqref{new1}. 

It follows from \cite{C-F-G-K1} that in this case, the function 
$\xi(\varphi_\pi,f_{\Delta(\tau,4),s})$ is given by
\begin{equation}\label{new20}
\xi(\varphi_\pi,f_{\Delta(\tau,4),s})(g)=\int\limits_{Sp_4(F)\backslash Sp_4({\bf A})}
\int\limits_{U_{4,16}(F)\backslash U_{4,16}({\bf A})}\varphi_\pi(h)
E(f_{\Delta(\tau,4),s})(u{}t(g,h))\psi^{-1}_{U_{4,16}}(u)dudh.
\end{equation}

As in equation \eqref{new1}, our goal is to compute the integral
\begin{equation}\label{new201}\notag
\int\limits_{Mat_2^0(F)\backslash Mat_2^0({\bf A})}\xi(\varphi_\pi,f_{\Delta(\tau,4),s})
\left (\begin{pmatrix} I_2&Z\\ &I_2\end{pmatrix} \right )\psi^{-1}(\text{tr}(TZ))dZ.
\end{equation}
It is equal to
\begin{multline}\label{new23}
\int\limits_{Mat_2^0(F)\backslash Mat_2^0({\bf A})}\int\limits_{Sp_4(F)\backslash Sp_4({\bf A})}
\int\limits_{U_{4,16}(F)\backslash U_{4,16}({\bf A})}\varphi_\pi(h)
E(f_{\Delta(\tau,4),s})(u\cdot {}t\left(\begin{pmatrix} I_2&Z\\ &I_2\end{pmatrix},h\right))\\
\psi^{-1}_{U_{4,16}}(u)\psi^{-1}(\text{tr}(TZ))dudhdZ.
\end{multline}
As explained in the introduction, after the statement of Theorem \ref{th1}, if we unfold this integral by first unfolding the Eisenstein series,  we will obtain an integral involving the functional \eqref{new1}. We conjecture that this integral will be Eulerian even-though this functional is not unique (locally at each place). However, our goal is to apply Fourier expansions and identities between Eisenstein series to obtain a "simpler" integral. More precisely, we will prove
\begin{theorem}\label{prop3}
Given $\phi\in \mathcal{S}(Mat_2({\bf A}))$, there are nontrivial choices of sections $f_{\Delta(\tau,4),s}$, $f'_{\Delta(\tau\otimes \chi_T,2),s}$ of $Ind_{Q_8({\bf A})}^{Sp_{16}({\bf A})}\Delta(\tau,4)|det\cdot|^s$, $Ind_{Q_4({\bf A})}^{Sp_8({\bf A})}\Delta(\tau\otimes\chi_T,2)|det\cdot|^s$, respectively, such that integral \eqref{new23} is equal to the integral
	\begin{equation}\label{new24}
	\int\limits_{Sp_{4}(F)\backslash Sp_{4}({\bf A})}
	\int\limits_{U_{2,8}(F)\backslash U_{2,8}({\bf A})}
	\varphi_\pi(h)
	\theta_{\psi,8}^\phi(l_T(v)i_T(1,h))E(f'_{\Delta(\tau\otimes\chi_T,2),s})(v(1,h))dvdh.
	\end{equation}
\end{theorem}

\begin{proof}
	The first step is to perform certain root exchanges in the inner $du$-integration inside integral \eqref{new23}. The process of root exchange was defined in general in \cite{G-R-S2} Section 7. In the notations of Lemma \ref{lemmix2}, let $B=U_{4,16}$; $Y$ is the subgroup of $B$ consisting of the elements \eqref{new5}, such that $b=Z=0$, and the coordinates of $a$ are zero, except those at the entries $(i,1)$, $(i,2)$, for $i=3,4$, which can be arbitrary; $C$ is the subgroup of $B$ generated by the root subgroups in $U_{4,16}$ which do not lie in $Y$; $X$ is generated by the root subgroups $I_{16}+x_{i,j}e'_{i,j}$, for $i=1,2$, $j=3,4$. Then $D=CX$. One can check that the first part of Lemma \ref{lemmix2} applies, and that the elements $t\left (\begin{pmatrix} I_2&Z\\ &I_2\end{pmatrix},h\right)$ in the integrand of \eqref{new23} belong to the set $\Omega$ in Lemma \ref{lemmix2}. 
	By the proof of the second part of Lemma \ref{lemmix2}, let $\tilde{f}_{\Delta(\tau,4),s}$ be a smooth, holomorphic section of $Ind_{Q_8({\bf A})}^{Sp_{16}({\bf A})}\Delta(\tau,4)|\text{det}\cdot|^s$, and let $\xi\in C_c^\infty(X({\bf A}))$. Take 
	$$
	f_{\Delta(\tau,4),s}(g)=\int\limits_{X({\bf A})} \xi(x)\tilde{f}_{\Delta(\tau,4),s}(gx)dx,\ g\in Sp_{16}(\bf A).
	$$
	Denote
	$$
	f'_{\Delta(\tau,4),s}(g)=\int\limits_{Y_1({\bf A})}\int\limits_{Y({\bf A})}\hat{\xi}(y)\tilde{f}_{\Delta(\tau,4),s}(gy)dy.
	$$
	Then the proof of the second part of Lemma \ref{lemmix2} shows that, for all $Z\in Mat_2^0(\bf A)$, $h\in Sp_4(\bf A)$,
	\begin{multline}\label{new24.3}
	\int\limits_{U_{4,16}(F)\backslash U_{4,16}({\bf A})} 
	E(f_{\Delta(\tau,4),s})(u\cdot t\left(\begin{pmatrix} I_2&Z\\ &I_2\end{pmatrix},h\right))\psi^{-1}_{U_{4,16}}(u)du=\\
	\int\limits_{D(F)\backslash D({\bf A})} 
	E(f'_{\Delta(\tau,4),s})(u\cdot t\left(\begin{pmatrix} I_2&Z\\ &I_2\end{pmatrix},h\right))\psi^{-1}_D(u)du.
	\end{multline}
	
	Let $w_0$ denote the Weyl element in $Sp_{16}(F)$ defined as
	$$w=\begin{pmatrix} I_2&&&\\ &&I_2&\\ &I_2&&\\ &&&I_2\end{pmatrix}\ \ \ \ 
	\ \ \ w_0=\begin{pmatrix} w&\\ &w^*\end{pmatrix}.$$
	On the r.h.s. of \eqref{new24.3}, we may replace in the Eisenstein series $u$ by $w_0u$. Now carry out the conjugation by $w_0$ in the integrand. Denote the right $w_0$-translate of $f'_{\Delta(\tau,4),s}$ by $f''_{\Delta(\tau,4),s}$. Note that the subgroup of all elements $w_0ut\left( \begin{pmatrix}I_2&Z\\&I_2\end{pmatrix},I_4\right)w_0^{-1}$, for $u\in D$, $Z\in Mat^0_2$, is equal to $U^0_{2^2,16}\hat{U}_{2,8}$, where $\hat{U}_{2,8}$ is the image of $U_{2,8}$ inside $Sp_{16}$ under the embedding $v\mapsto diag(I_4,v,I_4)$. On Adele points, the character $\psi_D(u)\psi(tr(TZ))$ goes to the character of $U^0_{2^2,16}({\bf A})\hat{U}_{2,8}({\bf A})$, which is trivial on $\hat{U}_{2,8}({\bf A})$ and on $U^0_{2^2,16}({\bf A})$ is equal to the character $\psi_{U^0_{2^2,16},T}$. Thus, substituting \eqref{new24.3} in \eqref{new23}, we get
	\begin{equation}\label{new25}
	\int\limits_{Sp_4(F)\backslash Sp_4({\bf A})}
	\int\limits_{U_{2,8}(F)\backslash U_{2,8}({\bf A})}
	\int\limits_{U^0_{2^2,16}(F)\backslash U^0_{2^2,16}({\bf A})}
	\varphi_\pi(h)
	E(f''_{\Delta(\tau,4),s})(u\hat{v}t(I_4,h),s)\psi^{-1}_{U^0_{2^2,16},T}(u)dudvdh.
	\end{equation}
		
	Note that for $h\in Sp_4$, $t(I_4,h)={}^d(I_2,diag(I_2,h,I_2))$. It is convenient to denote this by $\hat{h}$. We will use this notation also for $diag(I_2,h,I_2)$. Factor the group $U^0_{2^2,16}$ as $U^0_{2^2,16}=U_0Y_0$, where in the notation right after \eqref{new50}, $U_0$ is the subgroup of $U^0_{2^2,16}$ such that $y_0=0$, and $Y_0$ is the subgroup of $U^0_{2^2,16}$, such that in \eqref{new50}, $Z=0$, and for all $1\le i\le 4$,  $a_i=0$. Denote by  
	$\psi_{U_0,T}$ the restriction of $\psi_{U^0_{2^2,16},T}$ to $U_0$, and by $\psi_{Y_0}$ the restriction of $\psi_{U^0_{2^2,16},T}$ to $Y_0$.
	
	Consider the integral
	\begin{equation}\label{new251}
	\int\limits_{U_0(F)\backslash U_0({\bf A})}
	E(f''_{\Delta(\tau,4),s})(u_0u{}^d(h_0,g_0),s)\psi^{-1}_{U_0,T}(u_0)du_0
	\end{equation}
	as a function of $u\in U_{2^2,16}({\bf A})$ and of $(h_0,g_0)\in SO_{T_0}({\bf A})\times Sp_8({\bf A})$. Let us apply the theorem of Ikeda \cite{I}, as we did in the previous sub-section. The relevant (Adelic) Heisenberg group here is $\mathcal{H}_{17}({\bf A})$, realized as $Mat_{2\times 8}({\bf A})\times {\bf A}$. Let $\phi_1, \phi_2\in \mathcal{S}(Mat_{2\times 4}{\bf A})$, and consider the function $\varphi_{\phi_1,\phi_2}$ on $\mathcal{H}_{17}(\bf A)$ as in \cite{I}, p. 621. Let $\check{f}_{\Delta(\tau,4),s}$ be a smooth, holomorphic section of $Ind_{Q_8({\bf A})}^{Sp_{16}({\bf A})}\Delta(\tau,4)|\text{det}\cdot|^s$. Construct the section $\check{f}^{\phi_1,\phi_2}_{\Delta(\tau,4),s}$ similar to \eqref{new7.2}. As in \eqref{new7.1}, Ikeda's theorem tells us the following. Fix in \eqref{new251} $u\in U_{2^2,16}({\bf A})$, $h_0\in SO_{T_0}({\bf A})$, $g_0\in Sp_8({\bf A})$. Then
	\begin{multline}\label{new26}
	\int\limits_{U_0(F)\backslash U_0({\bf A})}
	E(\check{f}^{\phi_1,\phi_2}_{\Delta(\tau,4),s})(u_0u{}^d(h_0,g_0),s)\psi^{-1}_{U_0,T}(u_0)du_0= \\
	\theta_{\psi,16}^{\phi_1}(l^0_T(u)i^0_T(h_0,g_0))
	\int\limits_{U_{2^2,16}(F)\backslash U_{2^2,16}({\bf A})}
	\overline{\theta_{\psi,16}^{\phi_2}(l^0_T(u)i^0_T(h_0,g_0))}E(\check{f}_{\Delta(\tau,4),s})(u{}^d(h_0,g_0))\psi^{-1}_{U_{2^2,16}}(u)du.
	\end{multline}
	
	It follows from Lemma \ref{lem15}, with $h_0=1$, that there is a smooth, meromorphic section $\lambda(\check{f}_{\Delta(\tau,4),s},\phi_2)$ of $Ind_{Q_4({\bf A})}^{Sp_8({\bf A})}\Delta(\tau\otimes\chi_T,2)|\text{det}\cdot|^s$, such that the r.h.s. of \eqref{new26} is equal 
	\begin{equation}\label{new27}
\theta_{\psi,16}^{\phi_1}(l^0_T(u)i^0_T(1,g_0))E(\lambda(f_{\Delta(\tau,4),s},\phi_2))(g_0).
	\end{equation}

Since we don't know that $f''_{\Delta(\tau,4),s}$ is a sum of sections of the form $\check{f}^{\phi_1,\phi_2}_{\Delta(\tau,4),s}$, let us start with a smooth, holomorphic section $\check{f}_{\Delta(\tau,4),s}$. We can find Schwartz functions $\xi_i\in \mathcal{S}(X({\bf A}))$, $1\leq i\leq r$, and smooth, holomorphic sections $f^{(i)}_{\Delta(\tau,4),s}$, such that 
\begin{equation}\label{new26.1}
\rho(w_0)\check{f}^{\phi_1,\phi_2}_{\Delta(\tau,4),s}=\sum_{i=1}^r\int\limits_{Y({\bf A})}\hat{\xi}_i(y)\rho(y)f^{(i)}_{\Delta(\tau,4),s}dy.
\end{equation}
Now define
\begin{equation}\label{new26.2}
f_{\Delta(\tau,4),s}=\sum_{i=1}^r\int\limits_{X({\bf A})} \xi_i(x)\rho(x)f^{(i)}_{\Delta(\tau,4),s}dx.
\end{equation}
For this choice of $f_{\Delta(\tau,4),s}$, the integral \eqref{new23} is equal to 
\begin{multline}\label{new26.3}
\int\limits_{Sp_4(F)\backslash Sp_4({\bf A})}
\int\limits_{U_{2,8}(F)\backslash U_{2,8}({\bf A})}
\int\limits_{Y_0(F)\backslash Y_0({\bf A})}
\varphi_\pi(h)
\theta_{\psi,16}^{\phi_1}(l^0_T(y)i^0_T(1,v\hat{h}))\\
E(\lambda(\check{f}_{\Delta(\tau,4),s},\phi_2))(v\hat{h})\psi_{Y_0}^{-1}(y)dydvdh.
\end{multline}

Using the action of the Weil representation it is not hard to check that the integral over $Y_0(F)\backslash Y_0({\bf A})$ is equal to $\theta_{\psi,8}^{\phi'_1}
	(l_T(v)i_T(1,h))$, where $\phi'_1\in \mathcal{S}(Mat_2({\bf A}))$ is obtained from $\phi_1$ by $\phi'_1(x)=\phi_1(J_2T_0^{-1},x)$, $x\in Mat_2({\bf A})$. From this the theorem follows.
	
\end{proof}

With this result we can state the following,
\begin{conjecture}\label{conjnew}
	The integral \eqref{new24} is Eulerian, and represents the tensor product $L$-function $L(\pi\times\tau,s)$.
\end{conjecture}

\section{An Example with a  Fourier-Jacobi mixed model for symplectic groups }\label{fj}
Denote, for an odd number  $n$,  $V_n=U_{1,n+1}$. This is the unipotent radical of the standard  parabolic subgroup of $Sp_{n+1}$, whose Levi part is isomorphic to $GL_1\times Sp_{n-1}$. Thus, $V_n$ is isomorphic to the Heisenberg group  ${\mathcal H}_n$. Recall that we denote this isomorphism by $i_n$.

Let $\pi$ denote an irreducible, automorphic, cuspidal representation of $Sp_4^{(2)}({\bf A})$. Assume that there is an irreducible, automorphic, cuspidal representation $\sigma$ of $SL_2({\bf A})$,  such that the integral
\begin{equation}\label{fj1}
\int\limits_{SL_2(F)\backslash SL_2({\bf A})}
\int\limits_{V_3(F)\backslash V_3({\bf A})}\varphi_\pi^{(2)}(v_3\hat{g})
\theta_{\psi^{-1},2}^\phi(i_3(v_3)g)\varphi_\sigma(g^\iota)dv_3dg
\end{equation}
is not zero for some choice of data. Here, $\hat{g}=diag(1,g,1)$. For the definition of $g^\iota$, see equation \eqref{id5}. Also, $\phi\in \mathcal{S}({\bf A})$. We say that $\pi$ has a (global) Fourier-Jacobi 9mixed) model with respect to $\sigma$ (and $\psi$). 

Let $\tau$ denote an irreducible, automorphic, cuspidal representation of $GL_2({\bf A})$. In \cite{G-J-R-S}, an integral representation is introduced which represents the (partial) standard $L$-function $L^S_\psi(\pi\times\tau,s)$. That integral unfolds to the Fourier-Jacobi coefficient given by integral \eqref{fj1}. The integral introduced in \cite{G-J-R-S} for this case is
\begin{equation}\label{fj2}
\int\limits_{Sp_4(F)\backslash Sp_4({\bf A})}
\int\limits_{V_5(F)\backslash V_5({\bf A})}
\varphi_\pi^{(2)}(g)\theta_{\psi,4}^{\phi'}(i_5(v_5)g)E_{\tau,\sigma}(v_5g,s)dv_5dg
\end{equation}
The Eisenstein series $E_{\tau,\sigma}(\cdot,s)$ was defined in sub-section \ref{not3} right before equation \eqref{id5}. Also, $\phi'\in \mathcal{S}({\bf A}^2)$.

The doubling integral which represents the above partial $L$-function was announced in \cite{C-F-G-K2}. Thus, in this case, we take, for $g\in Sp^{(2)}_4({\bf A})$,
\begin{equation}\label{fj3}
\xi^{(2)}(\varphi_\pi,f_{\Delta(\tau,4)\gamma_\psi,s})(g)=\int\limits_{Sp_4(F)\backslash Sp_4({\bf A})}
\int\limits_{U_{4,16}(F)\backslash U_{4,16}({\bf A})}\varphi^{(2)}_\pi(h)
E^{(2)}(f_{\Delta(\tau,4)\gamma_\psi,s})(ut(g,h))\psi^{-1}_{U_{4,16}}(u)dudh.
\end{equation}
Here, $E^{(2)}(f_{\Delta(\tau,4)\gamma_\psi,s})$ is an Eisenstein series corresponding to a smooth, holomorphic section $f_{\Delta(\tau,4)\gamma_\psi,s}$ of $Ind_{Q^{(2)}_8({\bf A})}^{Sp^{(2)}_{16}({\bf A})}\Delta(\tau,4)\gamma_\psi |\text{det}\cdot|^s$.
Therefore, we want to consider the integral 
\begin{equation}\label{fj4}
\int\limits_{SL_2(F)\backslash SL_2({\bf A})}
\int\limits_{V_3(F)\backslash V_3({\bf A})}\xi^{(2)}(\varphi_\pi,f_{\Delta(\tau,4)\gamma_\psi,s})(v_3\hat{g})
\theta_{\psi^{-1},2}^\phi(i_3(v_3)g)\varphi_\sigma(g^\iota)dv_3dg,
\end{equation}
and we prove
\begin{theorem}\label{prop4}
There is a nontrivial choice of data such that integral \eqref{fj4} is equal to 
integral \eqref{fj2}.
\end{theorem}
\begin{proof}
The proof of this theorem is very similar to the proof of Theorem \ref{prop3}. We start with the inner $dv_3$-integration of \eqref{fj4}, that is we start with
\begin{multline}\label{fj5}
\int\limits_{U_{4,16}(F)\backslash U_{4,16}({\bf A})}
E^{(2)}(f_{\Delta(\tau,4)\gamma_\psi,s})(ut(v_3\hat{g},h))\psi^{-1}_{U_{4,16}}(u)du.
\end{multline}
Here, we identify $\hat{g}$ with $(\hat{g},1)$ or any fixed element projecting to $\hat{g}$. Next we perform certain root exchanges in \eqref{fj5}. We will do this three times. In the notations of Lemma \ref{lemmix2}, let $B_1=U_{4,16}$; $Y_1$ is the subgroup of $B_1$ consisting of the matrices of the form \eqref{new5}, such that all coordinates of $b, Z$ are zero, and all coordinates of $a$ are zero except $a_{4,1}, a_{4,2}$; $C_1$ is the subgroup of $B_1$ generated by its root subgroups which do not lie in $Y_1$; $X_1$ is the subgroup of elements $I_{16}+x_1e'_{1,4}+x_2e'_{2,4}$. Let $D_1=C_1X_1$. With these definitions  we can apply the first part of Lemma \ref{lemmix2} (with $B=B_1$, $C=C_1$, etc.). Next, let $B_2=D_1$; $Y_2$ is the root subgroup of $B_1$ consisting of the elements $I_{16}+ye'_{4,11}$. (In the notation of \eqref{new5}, $a=Z=0$, and all coordinates of $b$ are zero, except $b_{4,3}$). Let $C_2$ be the subgroup of $B_2$ generated by its root subgroups which do not lie in $Y_2$; $X_2$ is the root subgroup consisting of the elements $I_{16}+xe'_{3,4}$. Let $D_2=C_2X_2$. 
With these definitions  we can apply the first part of Lemma \ref{lemmix2} (with $B=B_2$, $C=C_2$, etc.). Finally, let $B_3=D_2$; $Y_3$ is the subgroup of $B_2$ consisting of the matrices of the form \eqref{new5}, with $b=Z=0$, and all coordinates of $a$ are zero, except $a_{2,1}, a_{3,1}$. Let $C_3$ be the subgroup of $B_3$ generated by its root subgroups which do not lie in $Y_3$; $X_3$ is the subgroup of elements $I_{16}+x_1e'_{1,2}+x_2e'_{1,3}$. Let $D_3=C_3X_3$. Again, we can apply the first part of Lemma \ref{lemmix2} (with $B=B_3$, etc.).
It follows that the integral \eqref{fj5} is equal to
\begin{equation}\label{fj5.1}
\int\limits_{Y_1({\bf A})}\int\limits_{Y_2({\bf A})}\int\limits_{Y_3({\bf A})}\int\limits_{D_3(F)\backslash D_3({\bf A})}
E^{(2)}(f_{\Delta(\tau,4)\gamma_\psi,s})(uy_3y_2y_1t(v_3\hat{g},h))\psi^{-1}_{D_3}(u)dudy_3dy_2dy_1.
\end{equation}
Let $Z_{4,1}$ the subgroup of $D_2$ consisting of the matrices \eqref{new5} with $a=b=0$, and all coordinates of $Z$ are zero, except $Z_{4,1}$. Denote by $Y'_1, Y''_1$ the subgroup of $Y_1$ such that, in the notation above, $a_{4,2}=0$, or $a_{4,1}=0$, respectively. Then $Y''_1Y_2Z_{4,1}$ is  group whose center is $Z_{4,1}$. Also $Y_1Y_3$ is a commutative subgroup, and the elements of $Y'_1Y_3$ commute with the elements of $Y_2Z_{4,1}$. Denote by $Y$ the quotient of $Y_1Y_2Y_3Z_{4,1}$ by $Z_{4,1}$. Denote $X=X_1X_2X_3$. This is a group; $X_1X_2$ is abelian and is normalized by $X_3$ (which is also abelian). Note that in \eqref{fj5.1}, $t(v_3\hat{g},h)$ normalizes $Y_1Y_2$ modulo $Z_{4,1}$, and normalizes $Y_3$ modulo $X_2$. Hence the integral \eqref{fj5.1} is equal to
\begin{equation}\label{fj5.2}
\int\limits_{Y({\bf A})}\int\limits_{D_3(F)\backslash D_3({\bf A})}
E^{(2)}(f_{\Delta(\tau,4)\gamma_\psi,s})(ut(v_3\hat{g},h)y)\psi^{-1}_{D_3}(u)dudy.
\end{equation}
Now, we can apply the proof of the third part of Lemma \ref{lemmix2} in two stages. The first is for $(D_2, Z_{4,1}\backslash Y_2Y_1Z_{4,1}, X_1X_2)$. Here, $t(v_3\hat{g},h)$ satisfies the three properties defining $\Omega'$ in Lemma \ref{lemmix2}, except that in the first property $t(v_3\hat{g},h)$ normalizes $Y_2({\bf A})Y_1({\bf A})$ only modulo $Z_{4,1}({\bf A})$, but this does not affect the argument. Similarly, for $(D_3, Y_3,X_3)$, $t(v_3\hat{g},h)$ satisfies the three properties defining $\Omega'$ in Lemma \ref{lemmix2}, except that in the first two properties $t(v_3\hat{g},h)$ normalizes $Y_3({\bf A})$, $X_3({\bf A})$ only modulo the root subgroup of elements $I_{16}+xe'_{1,4}$, which is a subgroup of $X_1({\bf A})$. Again, this does not affect the argument. Denote $X_{1,2}=X_1X_2$, $Y_{1,2}=Z_{4,1}\backslash Y_2Y_1Z_{4,1}$. Thus, write 
$$
f_{\Delta(\tau,4)\gamma_\psi,s}=\sum_{i=1}^r\sum_{j=1}^\ell \int\limits_{X_{1,2}({\bf A})}\int\limits_{X_3({\bf A})}\xi_i(x_{1,2})\eta_j(x_3)\rho(x_{1,2}x_3)f^{(i,j)}_{\Delta(\tau,4)\gamma_\psi,s}dx_3dx_{1,2},
$$
where, $\rho$ denotes a right translation, and for $1\leq i\leq r$, $1\leq j\leq \ell$, $\xi_i\in C_c^\infty(X_{1,2}({\bf A}))$, $\eta_j\in C_c^\infty(X_3({\bf A}))$, and $f^{(i,j)}_{\Delta(\tau,4)\gamma_\psi,s}$ are smooth, holomorphic sections of $Ind^{Sp_{16}^{(2)}({\bf A})}_{Q^{(2)}_8({\bf A})}\Delta(\tau,4)\gamma_\psi |\text{det}\cdot|^s$. Then repeating the proof of the third part of Lemma \ref{lemmix2} twice, we get that \eqref{fj5.2} and hence \eqref{fj5} is equal to
\begin{equation}\label{fj5.3}
\int\limits_{D_3(F)\backslash D_3({\bf A})}
E^{(2)}(f'_{\Delta(\tau,4)\gamma_\psi,s})(ut(v_3\hat{g},h))\psi^{-1}_{D_3}(u)du,
\end{equation}
where
\begin{equation}\label{fj5.4}
f'_{\Delta(\tau,4)\gamma_\psi,s})=\sum_{i=1}^r\sum_{j=1}^\ell \int\limits_{Y_{1,2}({\bf A})}\int\limits_{Y_3({\bf A})}\hat{\xi}_i(y_{1,2})\hat{\eta}_j(y_3)\rho(y_{1,2}y_3)f^{(i,j)}_{\Delta(\tau,4)\gamma_\psi,s}dy_3dy_{1,2}.
\end{equation}
We now bring in the $dv_3$-integration, and so we consider
\begin{equation}\label{fj5.5}
\int\limits_{D_3(F)\backslash D_3({\bf A})}
\int\limits_{V_3(F)\backslash V_3({\bf A})}
\theta_{\psi^{-1},2}^\phi(i_3(v_3)g)E^{(2)}(f'_{\Delta(\tau,4)\gamma_\psi,s})(ut(v_3\hat{g},h))\psi^{-1}_{D_3}(u)dudv_3.
\end{equation}

Consider the following Weyl element $w_0=\begin{pmatrix} w&\\ &w^*\end{pmatrix}$, where $w$ is a permutation matrix in $GL_8$ such that
$w_{i,j}=1$ at the $(1,1), (2,5), (3,2), (4,3)$ $(5,6), (6,4), (7,7), (8,8)$ entries. Conjugate by $w_0$ inside the Eisenstein series in \eqref{fj5.5}, that is write
$$
E^{(2)}(f'_{\Delta(\tau,4)\gamma_\psi,s})(ut(v_3\hat{g},h))=E^{(2)}(f'_{\Delta(\tau,4)\gamma_\psi,s})(w_0ut(v_3,1)w_0^{-1}t^{w_0}(\hat{g},h)w_0),
$$ 
where $t^{w_0}(\hat{g},h)=w_0t(\hat{g},h)w_0^{-1}$. Carry out this conjugation in \eqref{fj5.5}, change variables, and denote $f''_{\Delta(\tau,4)\gamma_\psi,s}=\rho(w_0)f'_{\Delta(\tau,4)\gamma_\psi,s}$. Then \eqref{fj5.5} becomes
\begin{multline}\label{fj6}
\int\limits_{v_5\in V_5(F)\backslash V_5({\bf A})}\int\limits_{v\in U_{2,12}(F)\backslash U_{2,12}({\bf A})}
\int\limits_{F^5\backslash {\bf A}^5}\int\limits_{u\in U^0_{1^2,16}(F)\backslash U^0_{1^2,16}({\bf A})} \theta_{\psi^{-1},2}^\phi((x_1,x_2,0)g)\\ 
E^{(2)}(f''_{\Delta(\tau,4)\gamma_\psi,s})(u\mu(x_1,x_2,a_1,a_2,a_3)\hat{v}i(g,v_5h)) \psi^{-1}_{U^0_{1^2,16},1}(u)\psi^{-1}_{U_{2,12}}(v)\psi^{-1}(a_1)dud(x,a)dvdv_5.
\end{multline}
Here, $(x_1,x_2,0)$ in the theta series is an element of $\mathcal{H}_3({\bf A})$, with zero coordinate in the center. For $v\in U_{2,12}$, $\hat{v}=diag(I_2,v,I_2)$. The element $v_5h$ is thought of as an element of $Sp_6({\bf A})$, with $h$ embedded in $Sp_6({\bf A})$ as $diag(1,h,1)$. Now, for $h'\in Sp_6$, and $g\in SL_2$,
\begin{equation}\label{fj6.1}
i(g,h')=diag(I_2,g,\begin{pmatrix}g_{1,1}&&g_{1,2}\\&I_6\\g_{2,1}&&g_{2,2}\end{pmatrix},g^*,I_2).
\end{equation}
Put in \eqref{fj6.1}, $i(g,h')=diag(I_2,t'(g,h'),I_2)$. Note that $t'(g,h)\in Sp_{12}$. The character $\psi_{U^0_{1^2,16},1}$ is defined as in Sec. \ref{not2} right before  \eqref{id1}, where we take $t=1$. Finally,
$$
\mu(x_1,x_2,a_1,a_2,a_3)=x_1e_{2,5}'+x_2e_{2,12}'+a_1e_{2,11}'
+a_2e_{2,13}'+a_3e_{2,14}'.
$$
The group of matrices $\mu(x_1,x_2,a_1,a_2,a_3)$ is a subgroup of $U_{1^2,16}$, and the projection $j_{13}$ of $U_{1^2,16}$ on ${\mathcal H}_{13}$ is injective on this subgroup. (See Sec. 2.1). We have 
$$
j_{13}(\mu(x_1,x_2,a_1,a_2,a_3))=(0_2,x_1,0_5,a_1,x_2,a_3,a_4,0)\in {\mathcal H}_{13}.
$$
Also, as mentioned above, we have  $u_2t'(v_5g,h)\in Sp_{12}$.

In integral \eqref{fj6}, consider the inner $du$- integration over $U_{1^2,16}^0(F)\backslash U_{1^2,16}^0({\bf A})$. Let us apply the theorem of Ikeda \cite{I}, as we did in the previous sub-sections. The relevant (Adelic) Heisenberg group here is $\mathcal{H}_{13}({\bf A})$. Let $\phi_1, \phi_2\in \mathcal{S}({\bf A}^6)$, and consider the function $\varphi_{\phi_1,\phi_2}$ on $\mathcal{H}_{13}(\bf A)$ as in \cite{I}, p. 621. Let $\check{f}_{\Delta(\tau,4)\gamma_\psi,s}$ be a smooth, holomorphic section of $Ind_{Q^{(2)}_8({\bf A})}^{Sp^{(2)}_{16}({\bf A})}\Delta(\tau,4)\gamma_\psi|\text{det}\cdot|^s$. Construct the section $\check{f}^{\phi_1,\phi_2}_{\Delta(\tau,4)\gamma_\psi,s}$ similar to \eqref{new7.2}. By Ikeda's theorem, the inner $du$-integral of \eqref{fj6}, with $\check{f}^{\phi_1,\phi_2}_{\Delta(\tau,4)\gamma_\psi,s}$ replacing $f''_{\Delta(\tau,4)\gamma_\psi,s}$, is equal to
\begin{multline}\label{fj7}
\theta_{\psi,12}^{\phi_1}((0_2,x_1,0_5,a_1,x_2,a_3,a_4,0)vt'(v_5g,h))\\
\int\limits_{U_{1^2,16}(F)\backslash U_{1^2,16}({\bf A})}
\overline{\theta_{\psi,12}^{\phi_2}(j_{13}(u)vt'(g,v_5h))}E^{(2)}(\check{f}_{\Delta(\tau,4)\gamma_\psi,s})(u{}vt'(g,v_5h))\psi^{-1}_1(u)du.
\end{multline}
Since we don't know that $f''_{\Delta(\tau,4)\gamma_\psi,s}$ is a sum of sections of the form $\check{f}^{\phi_1,\phi_2}_{\Delta(\tau,4)\gamma_\psi,s}$, we can use a similar argument as in the last sub-section and construct, for a given smooth, holomorphic section $\check{f}_{\Delta(\tau,4)\gamma_\psi,s}$, another such section $f_{\Delta(\tau,4)\gamma_\psi,s}$, similar to \eqref{new26.2}, such that the $du$-inner integral in \eqref{fj6} is equal to \eqref{fj7}.

 By Identity \eqref{id2}, the integral in \eqref{fj7} is equal to
$E(\Lambda(\check{f}_{\Delta(\tau,4)\gamma_\psi,s},\phi_2))(vt'(g,v_5h))$, where $\Lambda(\check{f}_{\Delta(\tau,4)\gamma_\psi,s},\phi_2)$ is a smooth, meromorphic section of $Ind_{Q_6({\bf A})}^{Sp_{12}({\bf A})}\Delta(\tau,3)|\text{det}\cdot|^s$. We denote, for short, $f_{\Delta(\tau,3),s}=\Lambda(\check{f}_{\Delta(\tau,4)\gamma_\psi,s},\phi_2)$. 
Plugging this into integral \eqref{fj6} we obtain the integral
\begin{multline}\label{fj8}
\int\limits_{v_5\in V_5(F)\backslash V_5({\bf A})}\int\limits_{v\in U_{2,12}(F)\backslash U_{2,12}({\bf A})}
\int\limits_{F^5\backslash {\bf A}^5}\theta_{\psi^{-1},2}^\phi((x_1,x_2,0)g)\\
\theta_{\psi,12}^{\phi_1}((0_2,x_1,0_5,a_1,x_2,a_3,a_4,0)vt'(g,v_5h))E(f_{\Delta(\tau,3),s})(vt'(g,v_5h))
\psi^{-1}_{U_{2,12}}(v)\psi^{-1}(a_1)d(x,a)dvdv_5.
\end{multline}

Consider the inner integration  over $a$.  We have
\begin{multline}\label{fj9} 
\int\limits_{(F\backslash {\bf A})^3}\theta_{\psi,12}^{\phi_1}((0_2,x_1,0_5,a_1,x_2,a_3,a_4,0)vt'(g,v_5h))\psi^{-1}(a_1)da=\\
\sum_{\xi_3,\xi_5,\xi_6\in F}\omega_{\psi,12}((0_2,x_1,0_6,x_2,0_3)vt'(g,v_5h))
\phi_1(0,0,\xi_3,1,\xi_5,\xi_6).
\end{multline}
It follows from the action of the Weil representation that the sum on the r.h.s. of \eqref{fj9} is left invariant under all $v\in U_{2,12}({\bf A})$. Hence, this sum is equal to 
\begin{equation}\label{fj10}
\sum_{\xi_3,\xi_5,\xi_6\in F}\omega_{\psi,8}((x_1,0_6,x_2,0)j(g,v_5h))\phi'_1(\xi_3,1,\xi_5,\xi_6).
\end{equation}
Here, for $x\in {\bf A}^4$, $\phi_1'(x)=\phi_1(0_2,x)$. Notice that the elements $t'(v_5g,h)$ in equation \eqref{fj9} are replaced by  $j(v_5g,h)\in Sp_8({\bf A})$ in equation \eqref{fj10}. See \eqref{id4} for the definition of $j(v_5g,h)$. Also, in the summation \eqref{fj10}, we have  $(x_1,0_6,x_2,0)\in {\mathcal H}_9(F)$. All this follows from the formulas of the action of the Weil representation. 

Take $\phi'_1$ of the form $\phi'_1(x_1,...,x_4)=\varphi_1(x_1)\varphi_2'(x_2,x_3,x_4)$, where $\varphi_1,\varphi'_2$ are Schwartz functions. Separating the summation in \eqref{fj10} over $\xi_3$ from the sum over $\xi_5,\xi_6$, it follows from the factorization properties of theta series, that the sum \eqref{fj10} is equal to 
$$
\theta_{\psi,2}^{\varphi_1}((x_1,x_2,0)g)\theta_{\psi,4}^{\varphi_2}(i_5(v_5)h),
$$
where $\varphi_2(y_1,y_2)=\varphi_2'(1,y_1,y_2)$. Thus, for the above choices of functions, integral \eqref{fj8} is equal to
\begin{multline}\label{fj11}
\int\limits_{v_5\in V_5(F)\backslash V_5({\bf A})}\int\limits_{v\in U_{2,12}(F)\backslash U_{2,12}({\bf A})}
\int\limits_{(F\backslash {\bf A})^2}\theta_{\psi^{-1},2}^\phi((x_1,x_2,0)g)\theta_{\psi,2}^{\varphi_1}((x_1,x_2,0)g)\\
\theta_{\psi,4}^{\varphi_2}(i_5(v_5)h)
E(f_{\Delta(\tau,3),s})(vt'(g,v_5h))
\psi^{-1}_{U_{2,12}}(v)dxdvdv_5.
\end{multline}
Consider the function $\phi\otimes \varphi_1\in \mathcal{S}({\bf A}^2)$, and view it as a function in the space of the Weil representation $\omega_{\psi,4}$ of $Sp^{(2)}_4({\bf A})\mathcal{H}_5({\bf A})$. Using the factorization properties of theta series, we can find a Weyl element $\gamma_0\in Sp_4(F)$, such that for the function $\Phi_{\phi,\varphi_1}=\omega_{\psi,4}(\gamma_0)(\phi\otimes \varphi_1)$
$$
\int\limits_{(F\backslash {\bf A})^2}\theta_{\psi^{-1},2}^\phi((x_1,x_2,0)g)\theta_{\psi,2}^{\varphi_1}((x_1,x_2,0)g)dx=\int\limits_{(F\backslash {\bf A})^2}\theta_{\psi,4}^{\Phi_{\phi,\varphi_1}}\left ((0_2,x_1,x_2,0)\begin{pmatrix} h&\\ &h^*\end{pmatrix}\right )dx.
$$
Unfolding the theta series, the right hand side of this equality is equal to $\Phi_{\phi,\varphi_1}(0_2)=\omega_{\psi,4}(\gamma_0)(\phi\otimes \varphi_1)(0_2)$. Thus, up to this value, the integral \eqref{fj11} is equal to
\begin{multline}\label{fj12}
\int\limits_{v_5\in V_5(F)\backslash V_5({\bf A})}\int\limits_{v\in U_{2,12}(F)\backslash U_{2,12}({\bf A})} 
\theta_{\psi,4}^{\varphi_2}(i_5(v_5)h)
E(f_{\Delta(\tau,3),s})(vt'(g,v_5h))
\psi^{-1}_{U_{2,12}}(v)dvdv_5.
\end{multline}
To summarize, we proved that for the choice of data described above, the inner unipotent integration in integral \eqref{fj4} is equal (up to multiplication by $\omega_{\psi,4}(\gamma_0)(\phi\otimes \varphi_1)(0_2)$) to integral \eqref{fj12}, and then the integral \eqref{fj4} is equal (up to $\omega_{\psi,4}(\gamma_0)(\phi\otimes \varphi_1)(0_2)$) to
\begin{multline}\label{fj13}
\int\limits_{Sp_4(F)\backslash Sp_4({\bf A})}
\int\limits_{V_5(F)\backslash V_5({\bf A})}
\varphi^{(2)}_\pi(h)\theta_{\psi,4}^{\varphi_2}(i_5(v_5)h)\\
\int\limits_{SL_2(F)\backslash SL_2({\bf A})}
\int\limits_{U_{2,12}(F)\backslash U_{2,12}({\bf A})} 
\varphi_\sigma(g^\iota)
E(f_{\Delta(\tau,3),s})(vt'(g,v_5h))
\psi^{-1}_{U_{2,12}}(v)dvdgdv_5dh.
\end{multline}
Now we apply Identity \eqref{id5}, which tells us, that there is a smooth, meromorphic section of $Ind_{Q_{2,2}({\bf A})}^{Sp_6({\bf A})}\tau |\text{det}\cdot|^s\times \sigma$, $\Lambda(f_{\Delta(\tau,3),s},\varphi_\sigma)$, such that the inner $dvdg$-integral in \eqref{fj13} is equal to $E(\Lambda(f_{\Delta(\tau,3),s},\varphi_\sigma))(v_5h)$. Thus, integral \eqref{fj13} is equal to 
$$
\int\limits_{Sp_4(F)\backslash Sp_4({\bf A})}
\int\limits_{V_5(F)\backslash V_5({\bf A})}
\varphi^{(2)}_\pi(h)\theta_{\psi,4}^{\varphi_2}(i_5(v_5)h)E(\Lambda(f_{\Delta(\tau,3),s},\varphi_\sigma))(v_5h)dv_5dh,
$$

and this is an integral of the form \eqref{fj2}.

\end{proof}

\end{document}